%% file: l-rom-na-arxiv.tex
\documentclass[preprint,12pt]{elsarticle}
\usepackage{amssymb, epsfig,amssymb, latexsym}
\usepackage{amsfonts,psfrag,amsmath,amsthm,bbm,color,url,cancel,multirow,lineno}
\usepackage[font=small,labelfont=bf]{caption}

\journal{J. Comput. Appl. Math.}

\input{notation}

\begin{document}

\begin{frontmatter}

\title{Numerical Analysis of the Leray Reduced Order Model}

 \author[xx]{Xuping Xie}
 \ead{xupingxy@vt.edu}
 \ead[url]{http://www.math.vt.edu/people/xupingxy}
 \author[dw]{David Wells}
 \ead{wellsd2@rpi.edu}
 \ead[url]{http://homepages.rpi.edu/~wellsd2}
 \author[zw]{Zhu Wang}
 \ead{wangzhu@math.sc.edu}
 \ead[url]{http://people.math.sc.edu/wangzhu}
 \author[xx]{Traian Iliescu\corref{cor1}}
 \ead{iliescu@vt.edu}
 \ead[url]{http://www.math.vt.edu/people/iliescu}
 \cortext[cor1]{corresponding author}
 \address[xx]{Department of Mathematics, Virginia Tech, Blacksburg, VA 24061, U.S.A.}
 \address[dw]{Department of Mathematical Sciences, Rensselaer Polytechnic Institute, Troy, NY 12180, U.S.A.}
 \address[zw]{Department of Mathematics, University of South Carolina, Columbia, SC 29208, U.S.A.}



\author{}

\address{}

\begin{abstract}
Standard ROMs generally yield spurious numerical oscillations in the simulation of convection-dominated flows.
Regularized ROMs use explicit ROM spatial filtering to decrease these spurious numerical oscillations.
The Leray ROM is a recently introduced regularized ROM that utilizes explicit ROM spatial filtering of the convective term in the Navier-Stokes equations.

This paper presents the numerical analysis of the finite element discretization of the Leray ROM.
Error estimates for the ROM differential filter, which is the explicit ROM spatial filter used in the Leray ROM, are proved.
These ROM filtering error estimates are then used to prove error estimates for the Leray ROM.
Finally, both the ROM filtering error estimates and the Leray ROM error estimates are numerically investigated in the simulation of the two-dimensional Navier-Stokes equations with an analytic solution. 
\end{abstract}


\end{frontmatter}


\clearpage 

\section{Introduction}

{\it Reduced order models (ROMs)} have been successfully used in the numerical simulation of structure-dominated fluid flows (see, e.g.,~\cite{ballarin2016fast,ballarin2015supremizer,bertagna2014model,bistrian2015improved,cordier2010calibration,ghommem2013mode,gouasmi2016characterizing,gunzburger2017ensemble,hesthaven2015certified,HLB96,kaiser2014cluster,noack2011reduced,perotto2017higamod,quarteroni2015reduced,san2015principal}).
Since they use a small number of carefully chosen basis functions (modes), ROMs can represent a computationally efficient alternative to standard numerical discretizations.
For convection-dominated flows, however, standard ROMs generally yield inaccurate results, usually in the form of spurious numerical oscillations (see, e.g.,~\cite{giere2015supg,wang2012proper}).
To mitigate these ROM inaccuracies, several numerical stabilization techniques have been proposed over the years (see, e.g.,~\cite{balajewicz2013low,balajewicz2016minimal,benosman2016robust,kalashnikova2010stability,osth2014need,wang2012proper,wang20162d}).
{\it Regularized ROMs (Reg-ROMs)} are recently proposed stabilized ROMs for the numerical simulation of convection-dominated flows~\cite{iliescu2017regularized,sabetghadam2012alpha,wells2017evolve}.
These Reg-ROMs use {\it explicit ROM spatial filtering} to smooth various ROM terms and thus increase the numerical stability of the resulting ROM.
This idea goes back to the great Jean Leray~\cite{leray1934sur}, who used it in the mathematical study of the Navier-Stokes equations (NSE).
In standard CFD, this idea was used to develop regularized models for the numerical simulation of turbulent flows~\cite{geurts2003regularization,layton2012approximate}).
In a ROM setting, a Reg-ROM was first used in~\cite{sabetghadam2012alpha} in the numerical simulation of the 1D Kuramoto-Sivashinsky equations.
A different Reg-ROM was proposed in~\cite{wells2017evolve} for the numerical simulation of the 3D NSE.
Reg-ROMs were also employed for the stabilization of ROMs in the numerical simulation of a stochastic Burgers equation~\cite{iliescu2017regularized}.

Reg-ROMs were successful in the numerical simulation of convection-dominated flows.
Two Reg-ROMs (the Leray ROM and the evolve-then-filter ROM) were used in the numerical simulation of a 3D flow past a circular cylinder at a Reynolds number $Re=1000$~\cite{wells2017evolve}.
These two Reg-ROMs produced accurate results in which the spurious numerical oscillations of standard ROMs were significantly decreased.
Despite the Reg-ROMs' success, to our knowledge there is no numerical analysis of the Reg-ROMs and the explicit ROM spatial filter used in their development.
In this paper, we take a first step in this direction and prove error estimates for the finite element discretization of (i) the {\it Leray ROM}~\cite{iliescu2017regularized,sabetghadam2012alpha,wells2017evolve}, which is a Reg-ROM; and (ii) the {\it ROM differential filter}, which is an explicit ROM spatial filter.

The rest of the paper is organized as follows:
In Section~\ref{sec:notation-preliminaries}, we present some notation and preliminaries.
In Section~\ref{sec:l-rom}, we present the ROM differential filter and the Leray ROM.
In Section~\ref{sec:error-analysis}, we prove error estimates for the ROM differential filter and the Leray ROM.
In Section~\ref{sec:numerical-results}, we verify numerically the error estimates proved in Section~\ref{sec:error-analysis}.
Finally, in Section~\ref{sec:conclusions}, we draw conclusions and outline possible future research directions.

\section{Notation and Preliminaries}
	\label{sec:notation-preliminaries}
	
We consider the numerical solution of the incompressible {\it Navier-Stokes equations (NSE)}:
\begin{equation}
\label{eq:nse}
\left\{ 
\begin{array}{cc}
\dfrac{\partial \bu}{\partial t} - \nu \Delta \bu + ( \bu \cdot \nabla )  \bu + \nabla p = \bff , & \text{ in } \Omega \times (0,T], \\
\nabla \cdot \bu = 0 , & \text{ in } \Omega \times (0,T] , \\
\bu = 0, & \text{ on } \partial \Omega \times (0,T] , \\
\bu(\bx, 0) = \bu^0(\bx), & \text{ in }  \Omega ,
\end{array} 
\right.
\end{equation}
where $\bu(\bx, t)$ and $p(\bx, t)$ represent the fluid velocity and pressure of a flow in the region $\Omega$, respectively, for $\bx\in \Omega$, $t\in [0, T]$, and $\Omega \subset \mathbb{R}^n$ with $n= 2$ or $3$; 
the flow is bounded by walls and driven by the force $\bff(\bx, t)$;  
$\nu$ is the reciprocal of the Reynolds number; 
and $\bu^0(\bx)$ denotes the initial velocity. 
We also assume that the boundary of the domain, $\partial \Omega$, is polygonal when $n= 2$ and is polyhedral when $n= 3$.

The following functional spaces and notations will be used in the paper:
\begin{equation*}
	\bX=\bH^1_0(\Omega) =\left\{\bv\in [L^2(\Omega)]^n: \nabla \bv \in [L^2(\Omega)]^{n\times n} \text{ and } \bv= {\bf 0} \text{ on } \partial \Omega \right\},
\end{equation*}	
\begin{equation*}
	Q=L^2_0(\Omega) = \left\{q\in L^2(\Omega): \int_{\Omega} q\, d {\bx} = 0 \right\}, 
\end{equation*}	
\begin{equation*}	
	\bV=\left\{\bv\in \bX: ( \nabla \cdot \bv, q)  = 0, \forall\, q \in Q \right\}, \text{ and  }
\end{equation*}	
\begin{equation*}	
	\bV^h=\left\{\bv_h \in \bX^h: ( \nabla \cdot \bv_h, q_h )  = 0, \forall\, q_h \in Q^h \right\}, 
\end{equation*}
where $\bX^h \subset \bX$ and $Q^h\subset Q$ are the finite element (FE) spaces of the velocity and pressure, respectively, and $h$ is the quasi-uniform mesh size.
We consider the div-stable pair of FE spaces $(\bX^h / Q^h) = (\PP^{m} / \PP^{m-1}), \, m \geq 2$ \cite{layton2008introduction}. 
We emphasize, however, that our analysis extends to more general FE spaces. 

Let $\cal H$ be a real Hilbert space endowed with inner product $(\cdot, \cdot)_{\cal H}$ and norm $\|\cdot\|_{\cal H}$. 
Let the trilinear form $b^*(\cdot, \cdot, \cdot)$ be defined as 
\begin{equation*}
b^*( \bu, \bv, \bw)  =\frac{1}{2}\left[ ( ( \bu\cdot\nabla ) \bv, \bw) - ( ( \bu\cdot\nabla ) \bw, \bv) \right] .
\label{eq:b*}
\end{equation*}

\begin{lemma}[see Lemma 13, Lemma 14 and Lemma 18 in \cite{layton2008introduction}]
\label{lem:b*}
For any functions $\bu, \, \bv, \bw \in \bX$, the skew-symmetric trilinear form $b^*(\cdot, \cdot, \cdot)$ satisfies
\begin{equation}
b^*(\bu, \bv, \bv) = 0,
\label{eq:b_skewsymm}
\end{equation}
\begin{equation} 
b^*(\bu, \bv , \bw)\leq C \|\nabla \bu\| \|\nabla \bv\| \|\nabla \bw\|, 
\label{eq:b_bound_0}
\end{equation}
and a sharper bound 
\begin{equation} 
b^*(\bu, \bv , \bw)\leq C \sqrt{\|\bu\| \|\nabla \bu\|}\|\nabla \bv\| \|\nabla \bw\|. 
\label{eq:b_bound_1}
\end{equation}
\end{lemma}

The weak formulation of the NSE \eqref{eq:nse} reads: 
Find $\bu \in \bX$ and $p \in Q$ such that
\begin{equation}
\left\{ 
	\begin{array}{cc}
		\left(  \dfrac{\partial \bu}{\partial t} , \bv \right) 
		+  \nu (  \nabla \bu , \nabla \bv ) 
		+ b^*( \bu,\bu,\bv) 
		- ( p , \nabla \cdot \bv )
		= ( \bff, \bv) , 
		& \quad \forall \, \bv \in \bX, \\
		( \nabla \cdot \bu , q ) 
		= 0,
		& \quad \forall \, q \in Q .
	\end{array} 
\right.
\label{eq:nse_weak}
\end{equation}
To ensure the uniqueness of the solution to \eqref{eq:nse_weak}, we make the following regularity assumptions (see Definition 29, Proposition 15, and Remark 10 in \cite{layton2008introduction}): 
\begin{assumption}
\label{assumption_regularity}
In \eqref{eq:nse}, we assume that 
$\bff\in L^2( 0, T; \bL^2(\Omega))$, 
$\bu^0\in \bV$, 
$\bu\in L^2( 0, T; \bX ) \bigcap L^{\infty} ( 0, T; \bL^2(\Omega) )$, 
$\nabla \bu \in ( L^4( 0, T; L^2(\Omega)) )^{n\times n}$, 
$\bu_t\in L^2( 0, T; \bX^*)$, and 
$p\in L^2 ( 0, T; Q )$. 
\end{assumption}

For any positive integer $M$, consider the time instances $t_k = k \, \Delta t, \, k= 0, \ldots, M$, where $\Delta t = T / M$.
Denote the solution of \eqref{eq:nse_weak} at time $t_k$ to be $\bu^{k} = \bu(t_k)$ and the force at $t_k$ to be $\bff^{k} = \bff(t_k)$, respectively. 
We make the following assumption, which will be used in Theorem~\ref{theorem_error}:
\begin{assumption}
	We assume that the solution of~\eqref{eq:nse_weak} satisfies the following stability estimate for all integers $\widetilde{M}$ such that $1 \leq \widetilde{M} \leq M$:
	\begin{equation}
		\| \bu^{\widetilde{M}} \|^2
		+ \Delta t \, \sum \limits_{k=0}^{\widetilde{M}-1} \| \nabla \bu^{k+1} \|^2
		\leq C \, ,
		\label{eqn:continuous-stability}
	\end{equation}
	where $C$ is a constant that does not depend on $\Delta t$, but can depend on the initial data.
	\label{assumption:continuous-stability}
\end{assumption}

For $k= 0, \ldots, M$, denote the FE approximate solution of \eqref{eq:nse_fe} at $t_k = k \, \Delta t$ to be $\bu_h^{k} = \bu_h(t_k)$. 
The FE semidiscretization of \eqref{eq:nse_weak} can be written as follows:
Find $\bu_h \in \bV^h$ such that
\begin{equation}
	\left(  \dfrac{\partial \bu_h}{\partial t} , \bv_h \right) 
	+  \nu (  \nabla \bu_h , \nabla \bv_h ) 
	+ b^*( \bu_h,\bu_h,\bv_h )
	= ( \bff, \bv_h ) , 
	\quad \forall \, \bv_h \in \bV^h
\label{eq:nse_fe}
\end{equation}
and 
$\bu_h(\cdot, 0) = \bu_h^0 \in \bV^h$.

\begin{assumption}[Finite Element Error]
\label{lem:fem}
We assume that the FE approximation $\bu_h$ of the full discretization of~\eqref{eq:nse_fe} satisfies the following error estimate:
\begin{equation}
\|\bu - \bu_h\| + h \| \nabla (\bu - \bu_h) \| \leq C (h^{m+1} + \Delta t). 
\label{eq:fem_err_u}
\end{equation}
We also assume the following standard approximation property: 
\begin{equation}
\inf\limits_{q_h \in Q^h}\|p- q_h\| \leq C h^m.
\label{eq:fem_err_p}
\end{equation} 
The constant $C$ in~\eqref{eq:fem_err_u}--\eqref{eq:fem_err_p} is a constant that does not depend on $h, m, \Delta t$, but can depend on the initial data.
\end{assumption}

\section{The Leray ROM (L-ROM)}
	\label{sec:l-rom}
	
In this section, we present the Leray ROM, which we will analyze in Section~\ref{sec:error-analysis}.
To this end, we present the standard ROM (Sections~\ref{sec:pod} and~\ref{sec:g-rom}) and the explicit ROM differential filter (Section~\ref{sec:rom-df}), which will be used to construct the Leray ROM (Section~\ref{sec:model}).

\subsection{Proper Orthogonal Decomposition}
	\label{sec:pod}

We briefly describe the POD method, following \cite{KV01}.
For a detailed presentation, the reader is referred to \cite{HLB96,volkwein2011model}.

Consider an ensemble of snapshots
$ \mathcal{R} := \mbox{span}\left\{ \bu(\cdot, t_0), \ldots, \bu(\cdot, t_M) \right\}$, which is a collection of velocity data from either numerical simulation results or experimental observations at time $t_i = i \, \Delta t$, $i=0, \ldots, M$.  
The POD method seeks a low-dimensional basis $\{ \bphi_1, \ldots, \bphi_r\}$ in $\cH$ 
that optimally approximates the snapshots, i.e., solves the minimization problem:
$
  \min \frac{1}{M+1} \sum_{\ell=0}^M 
  \left\| \bu(\cdot, t_{\ell}) - 
  \sum_{j=1}^r \left( \bu(\cdot, t_{\ell}),\bphi_j(\cdot) \right)_{\cH} \, \bphi_j(\cdot) 
  \right\|_{\cH}^2 
$
subject to the conditions $(\bphi_j, \bphi_i)_{\cH} = \delta_{ij}, \ 1 \leq i, j \leq r$, where $\delta_{ij}$ is the Kronecker delta. 
To solve this minimization problem, one can consider the eigenvalue problem 
$
K \, \bz_j = \lambda_j \, \bz_j,  \text{ for }j=1, \ldots, r,
$
where $K \in \R^{(M+1) \times (M+1)}$ 
is the snapshot correlation matrix with entries  
$\displaystyle K_{k\ell} = \frac{1}{M+1} \left( \bu(\cdot, t_{\ell}) , \bu(\cdot, t_k) \right)_{\cH}$ 
for $\ell, k = 0, \ldots, M$,
$\bz_j$ is the $j$-th eigenvector, and 
$\lambda_j$ is the associated eigenvalue. 
The eigenvalues are positive and sorted in descending order $\lambda_1\geq \ldots \geq \lambda_d > 0$, where $d$ is the rank of $\mathcal{R}$. 
It can then be shown that the POD basis functions are given by
$
\bphi_{j}(\cdot) = \frac{1}{\sqrt{\lambda_j}} \, \sum_{{\ell}= 0}^{M} (\bz_j)_{\ell} \, \bu(\cdot , t_{\ell}),
\  1 \leq j \leq r,
$
where $(\bz_j)_{\ell}$ is the $\ell$-th component of the eigenvector $\bz_j$. 
It can also be shown that the following error formula holds \cite{HLB96,KV01}:
\begin{equation}
\frac{1}{M+1} \sum_{{\ell}= 0}^M \left\| \bu(\cdot,t_{\ell}) - 
  \sum_{j=1}^r \left( \bu(\cdot,t_{\ell}),\bphi_j(\cdot) \right)_{\cH} \, \bphi_j(\cdot) 
  \right\|_{\cH}^2
= \sum_{j=r+1}^{d} \lambda_j \, .
\label{pod_error_formula}
\end{equation}
We define the ROM space as 
$
\bX^r := \text{span}\left\{\bphi_1, \ldots, \bphi_r\right\} 
$.

\begin{remark}
\label{rem:pod}
Since, the POD basis functions are linear combinations of the snapshots, 
the POD basis functions satisfy the boundary conditions in \eqref{eq:nse} and are solenoidal.  
If the FE approximations are used as snapshots, the POD basis functions belong to $\bV^h$, which yields  $\bX^r\subset \bV^h$.
\end{remark}

\subsection{The Galerkin ROM (G-ROM)}
	\label{sec:g-rom}

The ROM employs both Galerkin truncation and Galerkin projection. 
The former yields an approximation of the velocity field by a linear combination of the truncated POD basis: 
\begin{equation}
\label{eq:pod}
  {\bu}( {\bf x}, t ) \approx {\bur} ( {\bf x}, t ) 
  \equiv \sum_{j=1}^r a_j ( t ) \bphi_j ( {\bf x} ),
\end{equation}
where 
$\left\{a_{j}( t) \right\}_{j=1}^{r}$ are the sought time-varying coefficients representing the POD-Galerkin trajectories. 
Note that $r\ll N$, where $N$ denotes the number of degrees of freedom in the full order model (e.g., the FE approximation).
Replacing the velocity $\bu$ with $\bur$ in the NSE \eqref{eq:nse}, using the Galerkin method, and projecting the resulting equations onto the ROM space $\bX^r$, 
one obtains the {\it Galerkin ROM (G-ROM)} for the NSE: Find $\bur\in\bX^r$ such that 
\begin{equation}
\left(  \frac{\partial \bu_{r}}{\partial t} , \bphi \right) 
+  \nu(  \nabla \bur , \nabla \bphi ) 
+ b^*( \bur, \bur, \bphi) 
= ( {\bf f}, \bphi),  \quad
\forall \, \bphi \in \bX^r
\label{eqn:g-rom}
\end{equation}
and $\bu_r(\cdot, 0) \in \bX^r$. 
In \eqref{eqn:g-rom}, the pressure term vanishes because all POD modes are solenoidal and satisfy the appropriate boundary conditions. 
The error analysis of the spatial and temporal discretizations of the G-ROM~\eqref{eqn:g-rom} was considered in \cite{chapelle2012galerkin,iliescu2014are,kostova2015error,KV01,KV02,LCNY08,singler2014new}. 
Despite its appealing computational efficiency, the G-ROM \eqref{eqn:g-rom} has generally been limited to laminar flows.  
To overcome this restriction, we consider the Leray ROM.

\subsection{ROM Differential Filter (DF)}
	\label{sec:rom-df}

To construct the Leray ROM, we use the {\it ROM differential filter (DF)}, which is an explicit ROM spatial filter:
\begin{definition}[ROM Differential Filter]
	$\forall \bv \in \bX$, let $\obr{v}$ be the unique element of $\bXr$ such that 
	\begin{equation}
		\delta^2 \, \left( \nabla \obr{v} , \nabla \bvr \right)
		+ \left( \obr{v} , \bvr \right)
		= \left( \bv , \bvr \right)
		\quad \forall \bvr \in \bXr \, .
		\label{eqn:rom-df}
	\end{equation}
\end{definition}
The differential filter was introduced in large eddy simulation by Germano~\cite{germano1986differential,germano1986differential-b}.
In a ROM setting, the DF~\eqref{eqn:rom-df} was first used in~\cite{sabetghadam2012alpha} and later extended in~\cite{iliescu2017regularized,wells2017evolve,xie2017approximate}.

\subsection{The Model}
	\label{sec:model}

We consider the {\it Leray reduced order model (L-ROM)}~\cite{sabetghadam2012alpha,wells2017evolve}, which is a regularized ROM: 
Find $\bur \in \bX^r$ such that
\begin{equation}
\left(  \frac{\partial \bu_{r}}{\partial t} , \bphi \right) 
+ \nu (  \nabla \bu_{r} , \nabla \bphi ) 
+ b^*( \obur, \bur, \bphi) 
= (\bff , \bphi),  \quad
\forall \, \bphi \in \bX^r,
\label{eqn:l-rom}
\end{equation}
where the initial condition is given by the $L^2$ projection of $\bu^0$ on $\bX^r$: 
$
\bu_r(\cdot, 0) := \sum_{j=1}^r (\bu^0, \bphi_j)\bphi_j.
$ 

We consider the full discretization of \eqref{eqn:l-rom}: 
We use the backward Euler method with a time step $\Delta t$ for the time integration and the FE space $\PP^{m}$ with $m\geq 2$ and a mesh size $h$ for the spatial discretization. 
For $k= 0, \ldots, M$, we denote the approximation solution of \eqref{eqn:l-rom} at $t_k = k\Delta t$ to be $\bur^{k} = \bu_{h, r}(t_k)$ and the force at $t_k$ to be $\bff^{k} = \bff(t_k)$, respectively. 
Note that we have dropped the subscript $``h"$ in $\bur^{k}$ for clarity of notation.    
The discretized L-ROM reads: 
Find $\bur^{k}\in \bX^r$ such that, $\forall \, \bphi \in \bX^r,\, \forall \, k = 0, \ldots, M-1$,
\begin{eqnarray}
\left(  \frac{\bur^{k+1}-\bur^{k}}{\Delta t} , \bphi \right) 
+ \nu (  \nabla \bur^{k+1} , \nabla \bphi ) 
+ b^*( \overline{\bur^{k+1}}^r, \bur^{k+1}, \bphi) 
= ( \bff^{k+1}, \bphi),  
\label{eqn:l-rom-be}
\end{eqnarray}
where the initial condition is 
$
	{\bf u}_r^0 = \sum_{j=1}^{r} ({\bf u}^0 , \bphi_j) \, \bphi_j \, .
$

\section{Error Analysis}
	\label{sec:error-analysis}

In this section, we present the error analysis for the L-ROM discretization \eqref{eqn:l-rom-be}. 
We take the FE solutions $\bu_h(\cdot, t_i)$, $i=0, \ldots, M$ as snapshots and choose $\mathcal{H} = \bL^2$ in the POD generation. 
The error source includes three main components: the spatial FE discretization error, the temporal discretization error, and the POD truncation error. 
We derive the error estimate in three steps: 
First, we gather some necessary assumptions and preliminary results in Section \ref{sec:preliminaries}. 
Then, we prove a ROM filtering error estimate in Section~\ref{sec:rom-filtering-error-estimates}.
Finally, we prove an L-ROM error estimate in Section \ref{sec:l-rom-error-estimates}. 

\subsection{Preliminaries}
	\label{sec:preliminaries} 

\begin{definition}[Generic Constant $C$]
	In what follows, $C$ and $C_j$, where $j$ is a positive integer, will denote generic constants that do not depend on $\delta, r, h, \Delta t, m, \bphi_j, \lambda_j$, but can depend on $\nu, \bu_0, \bff, \bur^0, n, T$.
\end{definition}

\begin{definition}[ROM Laplacian]
	Let
	\begin{equation}
		\Deltar : \bX \rightarrow \bXr
		\label{eqn:pod-filtering-4}
	\end{equation}
	such that, $\forall \bv \in \bX, \Deltar \bv$ is the unique element of $\bXr$ such that 
	\begin{equation}
		\left( \Deltar \bv , \bvr \right)
		= - \left( \nabla \bv , \nabla \bvr \right)
		\quad \forall \bvr \in \bXr \, .
		\label{eqn:pod-filtering-5}
	\end{equation}
	\label{definition:pod-filtering-1}
\end{definition}

	\vspace*{-0.4cm}

We list a POD inverse estimate, which will be used in what follows.
Let $S_{r} \in \R^{r \times r}$ with $(S_r)_{i j} = (\nabla \varphi_j , \nabla \varphi_i)_{L^2}$ be the POD stiffness matrix.
Let $\| \cdot \|_2$  denote the matrix 2-norm.
\begin{lemma}[POD Inverse Estimates]
For all $\bvr \in \bX^r$, the following POD inverse estimate holds:
\begin{eqnarray}
	\| \nabla \bvr \|_{L^2}  
	&\leq& \CinvNabla \, \| \bvr \|_{L^2} 
	\label{eqn:pod-inverse-estimates-1} \, ,
\end{eqnarray}
where $\CinvNabla := \sqrt{\| S_r \|_2}$.
	\label{lemma:pod-inverse-estimates-1}
\end{lemma}
The inverse estimate~\eqref{eqn:pod-inverse-estimates-1} was proved in Lemma 2 and Remark 2 in~\cite{KV01} and was numerically investigated in Remark 3.3 in~\cite{iliescu2014are} and in Remark 3.2 in~\cite{giere2015supg}.

\begin{definition}[ROM $L^2$ Projection]
	Let
	\begin{equation}
		P_r : \bL^2 \rightarrow \bXr
		\label{eqn:rom-L2-projection-1}
	\end{equation}
	such that, $\forall \, \bv \in \bL^2, P_r(\bv)$ is the unique element of $\bXr$ such that 
	\begin{equation}
		\left( P_r(\bv) , \bvr \right)
		= \left( \bv , \bvr \right)
		\quad \forall \bvr \in \bXr \, .
		\label{eqn:rom-L2-projection-2}
	\end{equation}
\end{definition}

\begin{proposition}[$L^2$ Stability of ROM $L^2$ Projection]
	\begin{eqnarray}
		\| P_r(\bv) \|
		&\leq& \| \bv \|
		\qquad \forall \, \bv \in \bL^2 \, .
		\label{eqn:rom-L2-projection-3} 
	\end{eqnarray}
	\label{proposition:rom-L2-projection-1}
\end{proposition}

\begin{proof}
By choosing $\bv_r := P_r(\bv)$ in~\eqref{eqn:rom-L2-projection-2} and using the Cauchy-Schwarz inequality, we can prove~\eqref{eqn:rom-L2-projection-3}.
\end{proof}

The following error estimate was proved in Lemma 3.3 in~\cite{iliescu2014variational}:
\begin{lemma}
For any $\bu^k \in \bX$, its $L^2$ projection, $\bw_r^k = P_r(\bu^k)$, satisfies the following error estimates:
\begin{eqnarray}
	   && \hspace*{-2.0cm}
	   \frac{1}{M+1} \sum_{k=0}^{M} \left\| \bu^k- \bw_r^k  \right\|^2 
	\leq C \left( h^{2m+2} + \Delta t^2 + \sum _{j = r+1}^{d} \lambda_j \right) ,
\label{eq:error_eta} \\
	   && \hspace*{-2.0cm}
	    \frac{1}{M+1} \sum_{k=0}^{M} \left\| \nabla ( \bu^k- \bw_r^k ) \right\|^2
	\leq C \biggl( h^{2m} + \|S_r\|_2 h^{2m+2}+ 
	\nonumber \\
	&& \hspace*{4.0cm}
	(1+\|S_r\|_2 )\Delta t^2 + \sum _{j = r+1}^{d}\|\bphi_j\|_1^2\, \lambda_j  \biggr). 
\label{eq:error_geta}
\end{eqnarray}
\label{lemma_proj}
\end{lemma} 

We assume the following estimates, which were also assumed in~\cite{iliescu2014are}:
\begin{assumption} 
For any $\bu^k \in \bX$, its $L^2$ projection, $\bw_r^k = P_r(\bu^k)$, satisfies the following error estimates:
\begin{eqnarray}
&& \hspace*{-0.5cm}
	    \left\|  \bu^k - \bw_r^k \right\|  \leq
C \, \left(
h^{m+1}
+ \Delta t
+\sqrt{\sum\limits_{j=r+1}^d \lambda_j}\,
\right) \, , 
\label{eq:error_eta_no_sum} \\
&& \hspace*{-0.5cm}
	     \left\|  \nabla \left( \bu^k - \bw_r^k \right) \right\|  \leq
C \, \biggl(
h^{m}   
+ \sqrt{\|S_r\|_2}h^{m+1}
+ \sqrt{1+\|S_r\|_2}\Delta t
+ \sqrt{\sum _{j = r+1}^{d} \|\bphi_j\|_1^2\, \lambda_j}\,
\biggr) \, . 
\nonumber \\
\label{eq:error_geta_no_sum}
\end{eqnarray}
\label{assumption_poderror}
\end{assumption}

\subsection{ROM Filtering Error Estimates}
	\label{sec:rom-filtering-error-estimates}
	
In this section, we present theoretical results for the DF~\eqref{eqn:rom-df}, which was the essential tool that we used in developing the Leray ROM~\eqref{eqn:l-rom}.
The main result in this section is the estimate for the ROM filtering error in Lemma~\ref{lemma:rom-filtering-error-estimates}.
To our knowledge, this is the first estimate for the ROM filtering error.
This estimate is an extension of the FE filtering error estimates proved in~\cite{dunca2015vreman,dunca2013mathematical,layton2008numerical}.
This ROM filtering error estimate is important for the L-ROM error analysis in Section~\ref{sec:l-rom-error-estimates}, since we use it to treat the nonlinear term in~\eqref{eq:error-lhs-56-4}. 

\begin{lemma}[ROM Filtering Error Estimates]
	For $\bu^{k} \in \bX$ and $\Delta \bu^{k} \in \bL^2$,
	\begin{eqnarray}
		&& \hspace*{-0.5cm} 
		\delta^2 \, \| \nabla (\bu^{k} - \overline{\bu^{k}}^r) \|^2
		+ \| \bu^{k} - \overline{\bu^{k}}^r \|^2
		\nonumber \\
		&& \hspace*{-0.5cm} 
		 \leq 
		\, C \,  
		\left(
			h^{2 m +2} + \Delta t^2 + \sum_{j=r+1}^{d} \lambda_j	
		\right)
		+ C \, \delta^4 \, \| \Delta \bu^{k} \|^2 
		\nonumber \\
		&&  \hspace*{-0.5cm} 
		+ \, C \, \delta^2 \, 
		\biggl( 
			h^{2 m}   
			+ \|S_r\|_2 \, h^{2 m+2}
			+ (1+\|S_r\|_2) \, \Delta t^2
			+\sum _{j = r+1}^{d} \|\bphi_j\|_1^2\, \lambda_j \,
		\biggr) \, .
		\label{eqn:pod-filtering-7}
	\end{eqnarray}
	\label{lemma:rom-filtering-error-estimates}
\end{lemma}

\begin{proof}
Using the definition of the DF~\eqref{eqn:rom-df}, we have
\begin{equation}
	\delta^2 \, \left( \nabla \overline{\bu^{k}}^r , \nabla \bvr \right)
	+ \left( \overline{\bu^{k}}^r , \bvr \right)
	= \left( \bu^k , \bvr \right)
	\quad \forall \bvr \in \bXr \, .
	\label{eqn:pod-filtering-9}
\end{equation}
Since $\Delta \bu^k \in \bL^2$ (by hypothesis), we also have
\begin{equation}
	\delta^2 \, \left( \nabla \bu^k , \nabla \bvr \right)
	+ \left( \bu^k , \bvr \right)
	= - \delta^2 \, \left( \Delta \bu^k , \bvr \right)
	+ \left( \bu^k , \bvr \right)
	\quad \forall \bvr \in \bXr \, .
	\label{eqn:pod-filtering-10}
\end{equation}
Subtracting~\eqref{eqn:pod-filtering-9} from \eqref{eqn:pod-filtering-10}, we get
\begin{equation}
	\hspace*{-0.3cm} 
	\delta^2 \, \left( \nabla ( \bu^k - \overline{\bu^{k}}^r ) , \nabla \bvr \right)
	+ \left( \bu^k - \overline{\bu^{k}}^r , \bvr \right)
	= - \delta^2 \, \left( \Delta \bu^k , \bvr \right)
	\quad \forall \bvr \in \bXr \, .
	\label{eqn:pod-filtering-11}
\end{equation}
We decompose the error $\be := \bu^k - \overline{\bu^{k}}^r$ as follows:
\begin{equation}
	\be
	= \lp \bu^k - \bwr \rp
	- \lp \overline{\bu^{k}}^r - \bwr \rp
	:= \bfeta - \bPhir \, ,
	\label{eqn:pod-filtering-12}
\end{equation}
where $\bwr \in \bXr$ is arbitrary.
Using~\eqref{eqn:pod-filtering-12} in~\eqref{eqn:pod-filtering-11}, we get: $\forall \bvr \in \bXr$
\begin{equation}
	\hspace*{-0.1cm} 
	\delta^2 \, \left( \nabla \bPhir , \nabla \bvr \right)
	+ \left( \bPhir , \bvr \right)
	=  \delta^2 \, \left( \nabla \bfeta , \nabla \bvr \right)
	+ \left( \bfeta , \bvr \right)	
	+ \delta^2 \, \left( \Delta \bu^k , \bvr \right) \, .
	\label{eqn:pod-filtering-13}
\end{equation}

To prove~\eqref{eqn:pod-filtering-7}, we let $\bvr := \bPhir \in \bXr$ in~\eqref{eqn:pod-filtering-13} and then use the Cauchy-Schwarz and Young inequalities:
\begin{eqnarray}
	&& \hspace*{-1.0cm} 
	\delta^2 \, \left \| \nabla \bPhir \right \|^2
	+ \left \| \bPhir \right \|^2
	= \delta^2 \, \left( \nabla \bfeta , \nabla \bPhir \right)
	+ \left( \bfeta , \bPhir \right)	
	+ \delta^2 \, \left( \Delta \bu^k , \bPhir \right)
	\nonumber \\
	&& \hspace*{-1.5cm} 
	\leq \,  \delta^2 \, \| \nabla \bfeta \| \, \| \nabla \bPhir \|
	+ \| \bfeta \| \, \| \bPhir \|
	+ \delta^2 \, \| \Delta \bu^k \| \, \| \bPhir \|
	\nonumber \\
	&& \hspace*{-1.5cm} 
	\leq \, \delta^2 \left( \frac{\| \nabla \bfeta \|^2}{2} + \frac{\| \nabla \bPhir \|^2}{2}\right)
	+ \left( \|  \bfeta \|^2 + \frac{\| \bPhir \|^2}{4}\right)
	+ \left( \delta^4 \, \| \Delta \bu^k \|^2 + \frac{\| \bPhir \|^2}{4}\right) \, .
	\label{eqn:pod-filtering-14}
\end{eqnarray}
Rearranging~\eqref{eqn:pod-filtering-14} yields
\begin{eqnarray}
	\frac{\delta^2}{2} \, \| \nabla \bPhir \|^2
	+ \frac{1}{2} \, \| \bPhir \|^2
	\leq \frac{\delta^2}{2} \, \| \nabla \bfeta \|^2  
	+ \| \bfeta \|^2 
	+ \delta^4 \, \| \Delta \bu^k \|^2 \, .
	\label{eqn:pod-filtering-15}
\end{eqnarray}
Using~\eqref{eqn:pod-filtering-15} and the triangle inequality yields
\begin{eqnarray}
	&& \delta^2 \, \| \nabla (\bu^{k} - \overline{\bu^{k}}^r) \|^2
	+ \| \bu^{k} - \overline{\bu^{k}}^r \|^2
	\nonumber \\
	&\leq& C  \inf_{\bwr \in \bXr}   
	\left( 
		\delta^2 \, \| \nabla (\bu^{k} - \bwr) \|^2
		+ \| \bu^{k} - \bwr \|^2		
	\right)
	+ C \, \delta^4 \, \| \Delta \bu^{k} \|^2 \, .
	\label{eqn:pod-filtering-15b}
\end{eqnarray}
Using Assumption~\ref{assumption_poderror}, we get
\begin{eqnarray}
	&&
	\hspace*{-1.5cm}
	\inf_{\bwr \in \bXr} \| \bu^k - \bwr \|^2		
	\leq \| \bu^k - P_r(\bu^k) \|^2
	\leq C \, \left( h^{2 m +2} + \Delta t^2 + \sum_{j=r+1}^{d} \lambda_j \right) .
	\label{eqn:pod-filtering-15c} \\
	&& 
	\hspace*{-1.5cm}
	\inf_{\bwr \in \bXr} \| \nabla (\bu^k - \bwr) \|^2		
	\leq \| \nabla (\bu^k - P_r(\bu^k)) \|^2
	\nonumber \\
	&\leq& C \, \left(
			h^{2 m}   
			+ \|S_r\|_2 \, h^{2 m+2}
			+ (1+\|S_r\|_2) \, \Delta t^2
			+\sum _{j = r+1}^{d} \|\bphi_j\|_1^2\, \lambda_j \,
		\right) .
	\label{eqn:pod-filtering-15d}
\end{eqnarray}
Plugging~\eqref{eqn:pod-filtering-15c}-\eqref{eqn:pod-filtering-15d} into~\eqref{eqn:pod-filtering-15b} proves~\eqref{eqn:pod-filtering-7}. \\

\end{proof}

\begin{remark}
	Lemma~\ref{lemma:rom-filtering-error-estimates} extends Lemma 2.12 in~\cite{layton2008numerical} from the FE setting to the ROM setting.
	We could have extended Lemma 2.4 in~\cite{dunca2013mathematical} instead of Lemma 2.12 in~\cite{layton2008numerical} since the former yields better $\delta$ scalings of the $H^1$ seminorm of the filtering error.
	We emphasize, however, that the proof of Lemma 2.4 in~\cite{dunca2013mathematical} uses the $H^1$ stability of the $L^2$ projection~\cite{bank1981optimal,ern2004theory}.
	To our knowledge, the $H^1$ stability of the $L^2$ projection has not been yet proven in a ROM setting.
	Thus, we decided to extend to the ROM setting Lemma 2.12 in~\cite{layton2008numerical}, which does not rely on the $H^1$ stability of the ROM $L^2$ projection.
	\label{remark:rom-filtering-error-estimates}
\end{remark}

In the following lemma, we prove the stability of the ROM filtered variables, which will be used to prove Theorem~\ref{theorem_error}.
This lemma extends Lemma 2.11 in~\cite{layton2008numerical} (see also Lemma 2.3 in~\cite{dunca2013mathematical}) from the FE case to the ROM case.

\begin{lemma}[Stability of ROM Filtered Variables]
For $\bv \in \bX$, we have
	\begin{eqnarray}
		\| \obr{v} \|
		&\leq& \| \bv \| 
		\label{eqn:pod-filtering-6b} \\
		\| \nabla \obr{v} \|
		&\leq& \sqrt{\| S_r \|_2} \ \| \bv \| \, .
		\label{eqn:pod-filtering-6c}
	\end{eqnarray}	
For $\bv \in \bXr$, we have
	\begin{eqnarray}
		\| \nabla \obr{v} \|
		&\leq& \| \nabla \bv \| \, .
		\label{eqn:pod-filtering-6d}
	\end{eqnarray}	
	\label{lemma:pod-filtering-0}
\end{lemma}

\begin{proof}
To prove~\eqref{eqn:pod-filtering-6b}, we let $\bvr=\obr{v}$ in~\eqref{eqn:rom-df}:
\begin{equation}
	\delta^2 \, \| \nabla \obr{v} \|^2
	+ \| \obr{v} \|^2
	= \left( \bv , \obr{v} \right) \, .
	\label{eqn:pod-filtering-6e}
\end{equation}
Applying the Cauchy-Schwarz inequality to the RHS of~\eqref{eqn:pod-filtering-6e}, we get~\eqref{eqn:pod-filtering-6b}. 

To prove~\eqref{eqn:pod-filtering-6c}, we use the POD inverse estimate~\eqref{eqn:pod-inverse-estimates-1} and~\eqref{eqn:pod-filtering-6b}:
\begin{eqnarray}
	\| \nabla \obr{v} \|
	\stackrel{\eqref{eqn:pod-inverse-estimates-1}}{\leq} 
	\sqrt{\| S_r \|_2} \ \| \obr{v} \| 
	\stackrel{\eqref{eqn:pod-filtering-6b}}{\leq} 
	\sqrt{\| S_r \|_2} \ \| \bv \| \, .
	\label{eqn:pod-filtering-6f}
\end{eqnarray}	

Finally, to prove~\eqref{eqn:pod-filtering-6d}, we let $\bvr=\Deltar \obr{v}$ in~\eqref{eqn:rom-df}:
\begin{equation}
	\delta^2 \, \biggl( \nabla \obr{v} , \nabla (\Deltar \obr{v}) \biggr)
	+ \biggl( \obr{v}  ,  \Deltar \obr{v} \biggr)
	= \biggl( \bv  ,  \Deltar \obr{v} \biggr) \, .
	\label{eqn:pod-filtering-6g}
\end{equation}
Using the definition of the ROM Laplacian~\eqref{eqn:pod-filtering-5} in~\eqref{eqn:pod-filtering-6g}, we get
\begin{equation}
	- \delta^2 \, \biggl( \Deltar \obr{v} , \Deltar \obr{v} \biggr)
	- \biggl( \nabla \obr{v}  ,  \nabla \obr{v} \biggr)
	= - \biggl( \nabla \bv  ,  \nabla \obr{v} \biggr) \, .
	\label{eqn:pod-filtering-6h}
\end{equation}
Applying the Cauchy-Schwarz inequality to the RHS of~\eqref{eqn:pod-filtering-6h}, we get~\eqref{eqn:pod-filtering-6d}. 

\end{proof}

\subsection{Leray ROM Error Estimates}
	\label{sec:l-rom-error-estimates}

In this section, we prove a stability estimate (Lemma~\ref{lemma:stability}) and an error estimate (Theorem~\ref{theorem_error}) for the L-ROM \eqref{eqn:l-rom-be}.

\begin{lemma}
The solution of \eqref{eqn:l-rom-be} satisfies the following bound: 
\begin{equation}
\left\|  \bur^{\widetilde{M}} \right\|^2  
+ \Delta t \sum\limits_{k=0}^{\widetilde{M}-1} \left\| \nabla \bur^{k+1} \right\|^2
\leq 
C 
\qquad \forall \,1 \leq \widetilde{M} \leq M-1 \, .
\label{stability:c}
\end{equation}
\label{lemma:stability}
\end{lemma}

\begin{proof}
Choosing $\bphi:=\bur^{k+1}$ in \eqref{eqn:l-rom-be} and noting that $b^*(\overline{\bur^{k+1}}^r, \bur^{k+1} , \bur^{k+1}) = 0$ by \eqref{eq:b_skewsymm}, we obtain
\begin{equation}
\left( {\bur^{k+1}-\bur^{k}} , \bur^{k+1} \right)
+ {\nu \Delta t} \left( \nabla \bur^{k+1} , \nabla \bur^{k+1} \right)
= {\Delta t} \left( {\bff}^{k+1}, \bur^{k+1} \right).  
\label{eq:egy1}
\end{equation}
Using the Cauchy-Schwarz and Young inequalities yields
\begin{eqnarray}
\frac{1}{2} \left\|  \bur^{k+1} \right\| ^2 - \frac{1}{2} \left\| \bur^{k} \right\|^2 
+ \nu \Delta t \left\| \nabla \bur^{k+1} \right\|^2
\leq {\Delta t} \left( {\bff}^{k+1}, \bur^{k+1} \right).  
\label{eq:th:stability_1}
\end{eqnarray}
Applying the Cauchy-Schwarz and Young inequalities in \eqref{eq:th:stability_1}, 
we get 
\begin{eqnarray}
\hspace*{-0.5cm}
\frac{1}{2} \left\|  \bur^{k+1} \right\| ^2 - \frac{1}{2} \left\| \bur^{k} \right\|^2 
+ \nu \Delta t \left\| \nabla \bur^{k+1} \right\|^2
\leq \frac{\Delta t}{2 \nu} \left\| {\bff}^{k+1} \right\|_{-1}^2 + \frac{\nu \Delta t}{2} \left\| \nabla \bur^{k+1} \right\|^2.
\label{stability:1} 
\end{eqnarray}
The stability estimate \eqref{stability:c} follows by summing \eqref{stability:1} from 0 to $\widetilde{M}-1$. \end{proof}

\begin{theorem}
Under the regularity assumption of the exact solution (Assumption~ \ref{assumption_regularity} and Assumption~\ref{assumption:continuous-stability}), the assumption on the FE approximation (Assumption \ref{lem:fem}), and the assumption on the ROM projection error (Assumption \ref{assumption_poderror}), the solution of the L-ROM \eqref{eqn:l-rom-be} satisfies the following error estimate: There exists $\Delta t^*>0$ such that the inequality
\begin{eqnarray}
	&& \left\| {\bf u}^M - {\bf u}_r^M \right\|^2
	+ \Delta t \, \sum \limits_{k = 0}^{M-1} \left\| \nabla \left( {\bf u}^{k+1} - {\bf u}_{r}^{k+1}\right) \right\|^2
	\nonumber \\
	&\leq& 
	C \, \cF\biggl( \delta, h, \Delta t, \| S_r \|_2, \{ \lambda_j \}_{j=r+1}^{d}, \{ \|\bphi_j\|_1 \}_{j=r+1}^{d}\biggr)
	\label{eq:theorem_error_1}
\end{eqnarray}
holds for all $\Delta t< \Delta t^*$, where 
\begin{eqnarray}
	&& \hspace*{-0.8cm}
	\cF\biggl( \delta, h, \Delta t, \| S_r \|_2, \{ \lambda_j \}_{j=r+1}^{d}, \{ \|\bphi_j\|_1 \}_{j=r+1}^{d}\biggr)
	\nonumber \\
	&& \hspace*{-0.8cm} 
	= \left(
h^{2m}   
+ \|S_r\|_2 h^{2m+2}
+ \left( 1+\|S_r\|_2 \right) \Delta t^2
+ \sum\limits_{j=r+1}^d \|\bphi_j\|_1^2\, \lambda_j\,
\right)
	\nonumber \\
&& \hspace*{-0.8cm} 
	+ \, \| S_r \|_2^{1/2}
	\left(
	h^{2m+2}   
	+ \Delta t^2
	+ \sum\limits_{j=r+1}^d \lambda_j  \ 
\right)
		+ \frac{1}{\delta} \, 
		\biggl(
			h^{2 m +2} + \Delta t^2 + \sum_{j=r+1}^{d} \lambda_j	
		\biggr)
	\nonumber \\
	&& \hspace*{-0.8cm} 
		+ \, \delta \, 
		\biggl( 
			h^{2 m}   
			+ \|S_r\|_2 \, h^{2 m+2}
			+ (1+\|S_r\|_2) \, \Delta t^2
			+\sum _{j = r+1}^{d} \|\bphi_j\|_1^2\, \lambda_j \,
		\biggr)
		+ \, \delta^3 .
	\label{eqn:theorem-error-1b}
\end{eqnarray}
\label{theorem_error}
\end{theorem}

\begin{proof}
We start by splitting the error into two terms: 
\begin{equation}
\bu^{k+1}-\bur^{k+1} = \left(\bu^{k+1} - \bw_r^{k+1} \right) - \left( \bur^{k+1} - \bw_r^{k+1} \right) = \bfeta^{k+1} - \bPhir^{k+1}, 
\label{eq:error}
\end{equation}
where 
\begin{equation}
	\bw_r^{k+1} := P_r(\bu^{k+1}) \, .
	\label{eqn:l-rom-error-estimates-1}
\end{equation}
The first term in~\eqref{eq:error}, $\bfeta^{k+1}=\bu^{k+1} - \bw_r^{k+1}$, represents the difference between $\bu^{k+1}$ and its $L^2$ projection on $\bX^r$, which has been bounded in Lemma~\ref{lemma_proj}. 
The second term, $\bPhir^{k+1}$, is the remainder. 
Next, we construct the error equation. We first evaluate the weak formulation of the NSE \eqref{eq:nse_weak} at $t=t^{k+1}$ and let $\bv = \bPhir^{k+1}$, then subtract the L-ROM \eqref{eqn:l-rom-be} from it. 
We obtain
\begin{eqnarray}
&& \hspace*{-0.5cm}
\left( \bu^{k+1}_t , \bPhir^{k+1} \right) - \left( \frac{\bur^{k+1}-\bur^k}{\Delta t} , \bPhir^{k+1} \right)
+ \nu \left( \nabla \bu^{k+1} - \nabla \bur^{k+1}, \nabla \bPhir^{k+1} \right)  \nonumber \\ 
&& \hspace*{-0.5cm}
 + b^*\left( \bu^{k+1},\bu^{k+1},\bPhir^{k+1} \right) - b^*\left( \overline{\bur^{k+1}}^r, \bur^{k+1}, \bPhir^{k+1} \right)
- \left( p , \nabla \cdot \bPhir^{k+1} \right)
= 0. \qquad
\label{eq:error-eqn}
\end{eqnarray}
By subtracting and adding the difference quotient term, $\left(\frac{\bu^{k+1}-\bu^k}{\Delta t}, \bPhir^{k+1} \right)$, in \eqref{eq:error-eqn}, and applying the decomposition \eqref{eq:error}, we have
\begin{eqnarray}
&& \hspace*{-0.5cm}
\left( \bu^{k+1}_t - \frac{\bu^{k+1}-\bu^k}{\Delta t}, \bPhir^{k+1} \right) 
+\frac{1}{\Delta t}\left( \bfeta^{k+1} - \bPhir^{k+1}, \bPhir^{k+1} \right)  \nonumber \\
&& \hspace*{-0.5cm}
-\frac{1}{\Delta t} \left( \bfeta^k - \bPhir^k, \bPhir^{k+1} \right)
+\nu \left( \nabla \left( \bfeta^{k+1}-\bPhir^{k+1}\right) , \nabla \bPhir^{k+1} \right)  \nonumber \\ 
&& \hspace*{-0.5cm}
+ b^*\left( \bu^{k+1},\bu^{k+1},\bPhir^{k+1} \right)
- b^* \left( \overline{\bur^{k+1}}^r , \bur^{k+1}, \bPhir^{k+1} \right)
 - \left( p , \nabla \cdot \bPhir^{k+1} \right)
= 0.\qquad 
\label{eq:error-eqn2}
\end{eqnarray}
Note that \eqref{eqn:rom-L2-projection-2} implies that $\left( \bfeta^k , \bPhir^{k+1} \right) = 0$ and $\left( \bfeta^{k+1}, \bPhir^{k+1} \right) = 0$. 
Letting $\br^k = \bu_t^{k+1} - \frac{\bu^{k+1}-\bu^k}{\Delta t}$, 
we obtain
\begin{eqnarray}
&&\frac{1}{\Delta t}\left( \bPhir^{k+1}, \bPhir^{k+1}\right)  -\frac{1}{\Delta t}\left( \bPhir^k, \bPhir^{k+1}\right) 
+\nu  \left( \, \nabla \bPhir^{k+1} , \nabla \bPhir^{k+1}  \right) \nonumber \\
&&= \left( \br^k, \bPhir^{k+1} \right) 
+ \nu \left( \nabla \bfeta^{k+1} , \nabla \bPhir^{k+1}  \right)  
+ b^*\left( \bu^{k+1},\bu^{k+1}, \bPhir^{k+1}\right)  
\nonumber \\ 
&& - b^*\left( \overline{\bur^{k+1}}^r , \bur^{k+1}, \bPhir^{k+1}\right) 
- \left( p , \nabla \cdot \bPhir^{k+1} \right)   \, .
\label{eq:error-eqn3}
\end{eqnarray}
We estimate the LHS of \eqref{eq:error-eqn3} by applying the Cauchy-Schwarz and Young inequalities:
\begin{eqnarray}
\text{LHS} &=& \frac{1}{\Delta t}\left\| \bPhir^{k+1} \right\| ^2 -\frac{1}{\Delta t}\left( \bPhir^{k} , \bPhir^{k+1} \right) 
+\nu\left\| \nabla \bPhir^{k+1}  \right\| ^2 \nonumber \\
&\geq& \frac{1}{2 \Delta t}\left(\left\| \bPhir^{k+1} \right\| ^2 - \left\| \bPhir^{k} \right\| ^2\right)+\nu\left\| \nabla \bPhir^{k+1}  \right\| ^2 \, .
\label{eq:error-rhs}
\end{eqnarray} 
Using~\eqref{eq:error-eqn3} and \eqref{eq:error-rhs}, we obtain
\begin{eqnarray}
&&\left\| \bPhir^{k+1} \right\| ^2 - \left\| \bPhir^{k} \right\| ^2
+2\nu\Delta t\left\| \nabla \bPhir^{k+1}  \right\| ^2 \nonumber \\
&\leq& 2\Delta t\left( \br^k, \bPhir^{k+1}  \right) 
+ 2\nu\Delta t \left( \nabla \bfeta^{k+1} , \nabla \bPhir^{k+1}   \right)  
+ 2 \Delta t\, b^*\left( \bu^{k+1},\bu^{k+1}, \bPhir^{k+1} \right)  \nonumber \\ 
&& 
- 2\Delta t\, b^*\left( \overline{\bur^{k+1}}^r, \bur^{k+1}, \bPhir^{k+1} \right) 
- 2\Delta t \, \left( p , \nabla \cdot \bPhir^{k+1}  \right)  \, .
\label{eq:ineqn1}
\end{eqnarray} 
Using the Cauchy-Schwarz and Young inequalities in \eqref{eq:ineqn1}, we get
\begin{equation}
\left( \br^k, \bPhir^{k+1}  \right) \leq
\left\| \br^k \right\|_{-1} \, \left\| \nabla \bPhir^{k+1} \right\|  \leq
C_1 \lp \left\| \br^k \right\|_{-1}^2 + \left\| \nabla \bPhir^{k+1} \right\| ^2 \rp \, ,
\label{eq:error-lhs-1} 
\end{equation}
\begin{equation}
\nu \left( \nabla \bfeta^{k+1} , \nabla \bPhir^{k+1}  \right) \leq
\nu \left\| \nabla \bfeta^{k+1} \right\| \, \left\| \nabla \bPhir^{k+1}  \right\|  \leq
 C_2 \lp \left\| \nabla \bfeta^{k+1} \right\| ^2 
+ \left\| \nabla \bPhir^{k+1}  \right\| ^2 \rp \, .
\label{eq:error-lhs-4} 
\end{equation}
The nonlinear terms in \eqref{eq:ineqn1} can be written as follows: 
\begin{eqnarray}
& & b^*\lp \bu^{k+1},\bu^{k+1}, \bPhir^{k+1} \rp  
- b^*\lp \overline{\bur^{k+1}}^r, \bur^{k+1}, \bPhir^{k+1} \rp  \nonumber \\
&=& b^*\lp \bu^{k+1},\bu^{k+1}, \bPhir^{k+1} \rp  
- b^*\lp \overline{\bu^{k+1}}^r, \bu^{k+1}, \bPhir^{k+1} \rp \nonumber \\
& & + \, b^*\lp \overline{\bu^{k+1}}^r, \bu^{k+1}, \bPhir^{k+1} \rp
- b^*\lp \overline{\bur^{k+1}}^r, \bur^{k+1}, \bPhir^{k+1} \rp  \nonumber
\end{eqnarray}
\begin{eqnarray}
&& \hspace*{-1.0cm}
= b^*\lp \bu^{k+1} - \overline{\bu^{k+1}}^r , \bu^{k+1} , \bPhir^{k+1} \rp 
+ \, b^*\lp \overline{\bu^{k+1}}^r , \bu^{k+1} , \bPhir^{k+1} \rp
\nonumber \\
& & \hspace*{-1.0cm} 
- b^*\lp \overline{\bur^{k+1}}^r , \bu^{k+1} , \bPhir^{k+1} \rp 
+ \, b^*\lp \overline{\bur^{k+1}}^r , \bu^{k+1} , \bPhir^{k+1} \rp
- b^*\lp \overline{\bur^{k+1}}^r, \bur^{k+1}, \bPhir^{k+1} \rp \nonumber \\
&&  \hspace*{-1.0cm}
= b^*\lp \bu^{k+1} - \overline{\bu^{k+1}}^r , \bu^{k+1} , \bPhir^{k+1} \rp \nonumber \\
& &  \hspace*{-1.0cm} 
+ \, b^*\lp \overline{\bu^{k+1} - \bur^{k+1}}^r , \bu^{k+1} , \bPhir^{k+1} \rp
+ b^*\lp \overline{\bur^{k+1}}^r , \bu^{k+1} - \bur^{k+1} , \bPhir^{k+1} \rp \nonumber \\
&&  \hspace*{-1.0cm}
= b^*\lp \bu^{k+1} - \overline{\bu^{k+1}}^r , \bu^{k+1} , \bPhir^{k+1} \rp 
+ \, b^*\lp \overline{\bfeta^{k+1}}^r , \bu^{k+1} , \bPhir^{k+1} \rp
\nonumber \\
& &  \hspace*{-1.0cm} 
- b^*\lp \overline{\bPhir^{k+1}}^r , \bu^{k+1} , \bPhir^{k+1} \rp 
+ \ b^*\lp \overline{\bur^{k+1}}^r , \bfeta^{k+1} , \bPhir^{k+1} \rp
- \cancelto{0}{b^*\lp \overline{\bur^{k+1}}^r , \bPhir^{k+1} , \bPhir^{k+1} \rp} \, ,
\nonumber \\
\label{eq:error-lhs-56} 
\end{eqnarray}
where in the last term we have used \eqref{eq:b_skewsymm}. 
In \eqref{eq:error-lhs-56}, we apply Lemma~\ref{lem:b*}, Lemma~\ref{lemma:pod-filtering-0}, Lemma~\ref{lemma:rom-filtering-error-estimates} and the Cauchy-Schwarz and Young inequalities:

\begin{eqnarray}
&& \hspace*{-0.8cm} 
b^*\left( \overline{\bur^{k+1}}^r , \bfeta^{k+1} , \bPhir^{k+1} \right)  
\stackrel{\eqref{eq:b_bound_1}}{\leq}
C \, \left\| \overline{\bur^{k+1}}^r \right\|^{1/2}  \, 
\left\| \nabla \overline{\bur^{k+1}}^r \right\|^{1/2} \,  
\left\| \nabla \bfeta^{k+1} \right\| \, 
\left\| \nabla \bPhir^{k+1} \right\| \nonumber \\
&& \hspace*{2.0cm} 
\stackrel{\eqref{eqn:pod-filtering-6b}, \eqref{eqn:pod-filtering-6d}}{\leq}
C \, \left\| \bur^{k+1} \right\|^{1/2}  \, 
\left\| \nabla \bur^{k+1} \right\|^{1/2} \,  
\left\| \nabla \bfeta^{k+1} \right\| \, 
\left\| \nabla \bPhir^{k+1} \right\| \nonumber \\
&& \hspace*{2.4cm} 
\leq C_3 \, \lp \, \left\| \bur^{k+1} \right\| \,  \left\| \nabla \bur^{k+1} \right\| \,  \left\| \nabla \bfeta^{k+1}\right\| ^2 
+  \left\| \nabla \bPhir^{k+1} \right\| ^2 \rp \, ;
\label{eq:error-lhs-56-1} 
\end{eqnarray}

\begin{eqnarray}
&& \hspace*{-0.8cm} 
b^*\left( \overline{\bfeta^{k+1}}^r , \bu^{k+1} , \bPhir^{k+1} \right)  
\stackrel{\eqref{eq:b_bound_1}}{\leq}
C \, \left\| \overline{\bfeta^{k+1}}^r \right\|^{1/2} \, 
\left\| \nabla \overline{\bfeta^{k+1}}^r \right\|^{1/2} \, 
\left\| \nabla \bu^{k+1}\right\|  \, 
\left\| \nabla \bPhir^{k+1} \right\|  \nonumber \\
&& \hspace*{2.0cm} 
\stackrel{\eqref{eqn:pod-filtering-6b}, \eqref{eqn:pod-filtering-6c}}{\leq}
C \, \left\| \bfeta^{k+1} \right\|^{1/2} \, 
\| S_r \|_2^{1/4} \, \left\| \bfeta^{k+1} \right\|^{1/2} \, 
\left\| \nabla \bu^{k+1}\right\|  \, 
\left\| \nabla \bPhir^{k+1} \right\|  \nonumber \\
&& \hspace*{2.4cm} 
\leq C_4 \, \lp \| S_r \|_2^{1/2} \, \left\| \bfeta^{k+1} \right\|^2 \, \left\| \nabla \bu^{k+1}\right\| ^2  + \left\| \nabla \bPhir^{k+1} \right\| ^2 \rp \, ;
\label{eq:error-lhs-56-2} 
\end{eqnarray}

\begin{eqnarray}
b^*\left( \overline{\bPhir^{k+1}}^r , \bu^{k+1} , \bPhir^{k+1} \right) 
&\stackrel{\eqref{eq:b_bound_1}}{\leq}&
C \, \left\| \overline{\bPhir^{k+1}}^r \right\| ^{1/2} \, 
\left\| \nabla \overline{\bPhir^{k+1}}^r \right\|^{1/2} \, 
\left\| \nabla \bu^{k+1}\right\| \, 
\left\| \nabla \bPhir^{k+1} \right\|  \nonumber \\
&\stackrel{\eqref{eqn:pod-filtering-6b}, \eqref{eqn:pod-filtering-6d}}{\leq}&
C \, \left\| \bPhir^{k+1} \right\| ^{1/2} \, 
\left\| \nabla \bPhir^{k+1} \right\|^{1/2} \, 
\left\| \nabla \bu^{k+1}\right\| \, 
\left\| \nabla \bPhir^{k+1} \right\|  \nonumber \\
&=& C \left\|  \bPhir^{k+1} \right\| ^{\half} \, \left\| \nabla \bu^{k+1}\right\|  \, \left\| \nabla \bPhir^{k+1} \right\| ^{\frac{3}{2}} \nonumber \\
&\leq&
C_5 \, \lp \left\| \nabla \bu^{k+1}\right\| ^4 \, \left\| \bPhir^{k+1} \right\| ^2  + 
 \left\| \nabla \bPhir^{k+1} \right\| ^2 \rp \, .
\label{eq:error-lhs-56-3} 
\end{eqnarray}

\begin{eqnarray}
	&& \hspace*{-1.4cm} 
	b^*\lp \bu^{k+1} - \overline{\bu^{k+1}}^r , \bu^{k+1} , \bPhir^{k+1} \rp 
	\nonumber \\
	&& \hspace*{-1.4cm} \stackrel{\eqref{eq:b_bound_1}}{\leq}
	C \, \lnorm \bu^{k+1} - \overline{\bu^{k+1}}^r \rnorm^{1/2} \, 
	\lnorm \nabla \lp \bu^{k+1} - \overline{\bu^{k+1}}^r \rp  \rnorm^{1/2} \, 
	\left\| \nabla \bu^{k+1}\right\|  \, 
	\left\| \nabla \bPhir^{k+1} \right\| 
	\nonumber \\
	&& \hspace*{-1.4cm} 
	\leq C_6 \, \lp
		\lnorm \bu^{k+1} - \overline{\bu^{k+1}}^r \rnorm \, 
		\lnorm \nabla \lp \bu^{k+1} - \overline{\bu^{k+1}}^r \rp  \rnorm \, 
		\left\| \nabla \bu^{k+1} \right\|^2
		+ \left\| \nabla \bPhir^{k+1} \right\|^2
	\rp ,
	\label{eq:error-lhs-56-4} 
\end{eqnarray}
which can be bounded by using Lemma~\ref{lemma:rom-filtering-error-estimates}.
Since $\bPhir^{k+1}  \in \bX^r \subset \bV^h$, the pressure term on the RHS of \eqref{eq:ineqn1} can be written as 
\begin{equation}
- \left( p , \nabla \cdot \bPhir^{k+1}  \right)
= - \left( p - q_h, \nabla \cdot \bPhir^{k+1}  \right) \, ,
\label{eq:error-lhs-pressure_1} 
\end{equation}
where $q_h$ is any function in $Q^h$.
Thus, the pressure term can be estimated as follows by using the Cauchy-Schwarz  and Young inequalities:
\begin{equation}
	- \left( p , \nabla \cdot \bPhir^{k+1}  \right)
	\leq C_7 \, \lp 
		\left\| p - q_h \right\|^2
		 + \left\| \nabla \bPhir^{k+1}  \right\|^2 
		 \rp \, .
	\label{eq:error-lhs-pressure_2} 
\end{equation}

Choosing $C_1-C_7$ appropriately,  then substituting inequalities~\eqref{eq:error-lhs-1}--\eqref{eq:error-lhs-4}, \eqref{eq:error-lhs-56-1}--\eqref{eq:error-lhs-56-3}, and~\eqref{eq:error-lhs-56-4} in \eqref{eq:ineqn1}, we obtain
\begin{eqnarray}
&& \hspace*{-1.0cm}
\left\| \bPhir^{k+1} \right\| ^2 - \left\| \bPhir^{k} \right\| ^2 + 
 C_8 \, \Delta t \left\| \nabla \bPhir^{k+1}  \right\| ^2 
\nonumber \\
&& \hspace*{-1.4cm}
\leq C_9 \biggl(
	\Delta t \, \left\| \br^k \right\|_{-1}^2 
	+ \Delta t \, \left\| \nabla \bfeta^{k+1} \right\| ^2          
	+ \Delta t \, \left\| \bur^{k+1} \right\| \, \left\| \nabla \bur^{k+1} \right\| \,  \left\| \nabla \bfeta^{k+1} \right\|^2
	\nonumber \\
	&& \hspace*{-1.4cm} 
	+ \Delta t \, \| S_r \|_2^{1/2} \, \left\| \nabla \bu^{k+1} \right\|^2  \, \left\| \bfeta^{k+1} \right\|^2    
	+ \Delta t \, \left\| \nabla \bu^{k+1}\right\| ^4 \, \left\| \bPhir^{k+1} \right\| ^2
	\nonumber \\
	&& \hspace*{-1.4cm} 
	+ \lnorm \bu^{k+1} - \overline{\bu^{k+1}}^r \rnorm \, 
		\lnorm \nabla \lp \bu^{k+1} - \overline{\bu^{k+1}}^r \rp  \rnorm \, 
		\left\| \nabla \bu^{k+1} \right\|^2
	+ \Delta t \, \left\| p - q_h \right\|^2 \biggr) \, .
\label{eq:error-eqn4}
\end{eqnarray}
Summing \eqref{eq:error-eqn4} from $k=0$ to $k=M-1$, we have 
\begin{eqnarray}
	&& \hspace*{-0.8cm}
	\left\| \bPhir^{M} \right\|^2 
	+ C_8 \, \Delta t \sum\limits_{k=0}^{M-1} \left\| \nabla \bPhir^{k+1}  \right\| ^2 
	\leq
	\left\| \bPhir^{0} \right\|^2 
	+ C_9 \, \Delta t \, \biggl(
	\sum\limits_{k=0}^{M-1} \left\| \br^k \right\|_{-1}^2 
	\nonumber \\
	&& \hspace*{-0.3cm} 
	+ \sum\limits_{k=0}^{M-1}  \left\| \bur^{k+1} \right\| \, \left\| \nabla \bur^{k+1} \right\| \,  \left\| \nabla \bfeta^{k+1} \right\|^2 
	+ \| S_r \|_2^{1/2} \, \sum\limits_{k=0}^{M-1} \left\| \nabla \bu^{k+1} \right\|^2  \, \left\|  \bfeta^{k+1} \right\|^2  
	\nonumber \\
	&& \hspace*{-0.3cm} 
	+ \sum\limits_{k=0}^{M-1} \left\| \nabla \bu^{k+1}\right\| ^4 \, \left\| \bPhir^{k+1} \right\| ^2 
	+ \sum\limits_{k=0}^{M-1} \left\| p - q_h \right\|^2  
	+ \sum\limits_{k=0}^{M-1} \left\| \nabla \bfeta^{k+1} \right\| ^2        
	\nonumber \\
	&& \hspace*{-0.3cm}
	+ \sum\limits_{k=0}^{M-1} \lnorm \bu^{k+1} - \overline{\bu^{k+1}}^r \rnorm \, 
		\lnorm \nabla \lp \bu^{k+1} - \overline{\bu^{k+1}}^r \rp  \rnorm \, 
		\left\| \nabla \bu^{k+1} \right\|^2
	\biggl) \, . 
	\label{eqn:error-eqn5}
\end{eqnarray}
The first term on the RHS of \eqref{eqn:error-eqn5} vanishes, since $\bu_r^0 = \bw_r^0$. 
By using the Poincar\'{e}-Friedrichs inequality, the second term on the RHS of \eqref{eqn:error-eqn5} can be estimated as follows (see, e.g., \cite{iliescu2014variational}): 
\begin{equation}
	\Delta t \, \sum\limits_{k=0}^{M-1} \left\| \br^k \right\|_{-1}^2 
	\leq 
	C_{10} \, \Delta t \, \sum\limits_{k=0}^{M-1} \left\| \br^k \right\|^2 
	\leq 
	C_{11} \, \Delta t^2 \, .
	\label{eq:proof1}
\end{equation}
Using \eqref{eq:error_geta}, the third term on the RHS of \eqref{eqn:error-eqn5} can be estimated as follows:
\begin{eqnarray}
&& \hspace*{-0.5cm}
\Delta t \sum \limits_{k=0}^{M-1} \| \nabla \bfeta^{k+1}\|^2
\leq C_{12} \, \biggl(
h^{2m}   
+ \|S_r\|_2 h^{2m+2}
+ \left( 1+\|S_r\|_2 \right) \Delta t^2
	\nonumber \\
	&& \hspace*{4.0cm}
+ \sum\limits_{j=r+1}^d \|\bphi_j\|_1^2\, \lambda_j\,
\biggr).
\label{eqn:3.42_rhs_4_and_8}
\end{eqnarray}
To estimate the fourth term on the RHS of \eqref{eqn:error-eqn5}, we use Lemma~\ref{lemma:stability} and Assumption~\ref{assumption_poderror}:
\begin{eqnarray}
	&& \hspace*{-0.2cm}
	\Delta t \sum \limits_{k=0}^{M-1} \left\| {\bf u}_{r}^{k+1} \right\| \, \left\| \nabla {\bf u}_{r}^{k+1} \right\| \, \left\| \nabla \bfeta^{k+1} \right\|^2
	\stackrel{\eqref{stability:c}}{\leq} 
	\widetilde{C}_{13} \, \Delta t \sum \limits_{k=0}^{M-1} \left\| \nabla {\bf u}_{r}^{k+1} \right\| \, \left\| \nabla \bfeta^{k+1} \right\|^2 
	\nonumber \\
	&& \hspace*{-1.0cm}
	\stackrel{\eqref{stability:c}}{\leq} 
	C_{13} \, \left\| \nabla \bfeta^{k+1} \right\|^2 
	\stackrel{\eqref{eq:error_geta_no_sum}}{\leq}
	C_{13} \,
	\Big(
		h^{2m}   
		+ \|S_r\|_2 h^{2m+2}
		+ \left( 1+\|S_r\|_2 \right) \Delta t^2 
	\nonumber \\
	&& \hspace*{4.0cm}
	+ \sum\limits_{j=r+1}^d \|\bphi_j\|_1^2\, \lambda_j\,
	\Big),
\label{eqn:3.42_rhs_5d}
\end{eqnarray}
where we used estimate \eqref{eq:error_geta_no_sum} in the derivation of \eqref{eqn:3.42_rhs_5d}.
Using \eqref{eq:error_geta} would not have been enough for the asymptotic convergence of \eqref{eqn:3.42_rhs_5d}.

Using Assumption~\ref{assumption_poderror}, the fifth term on the RHS of \eqref{eqn:error-eqn5} can be bounded as follows:
\begin{eqnarray}
&& \Delta t \, \| S_r \|_2^{1/2} \, \sum \limits_{k=0}^{M-1} \left\| \nabla {\bf u}^{k+1} \right\|^2 \, \left\| \bfeta^{k+1} \right\|^2
\nonumber \\
&\stackrel{\eqref{eq:error_eta_no_sum} }{\leq}&
\widetilde{C}_{14} \, \Delta t \, \| S_r \|_2^{1/2} \, \sum \limits_{k=0}^{M-1} \left\| \nabla {\bf u}^{k+1} \right\|^2 \, 
\left(
	h^{m+1}   
	+ \Delta t
	+ \sqrt{\sum\limits_{j=r+1}^d \lambda_j } \ 
\right)^2
\nonumber \\
&\stackrel{\eqref{eqn:continuous-stability}}{\leq}&
C_{14} \, \, \| S_r \|_2^{1/2} \, 
\left(
	h^{m+1}   
	+ \Delta t
	+ \sqrt{\sum\limits_{j=r+1}^d \lambda_j } \ 
\right)^2
\, .
\label{eqn:3.42_rhs_6}
\end{eqnarray}
To estimate the seventh term on the RHS of \eqref{eqn:error-eqn5}, we use Lemma~\ref{lemma:rom-filtering-error-estimates}:
\begin{eqnarray}
	&& \hspace*{-0.3cm}
	\Delta t \, \sum\limits_{k=0}^{M-1} 
		\lnorm \bu^{k+1} - \overline{\bu^{k+1}}^r \rnorm \, 
		\lnorm \nabla \lp \bu^{k+1} - \overline{\bu^{k+1}}^r \rp  \rnorm \, 
		\left\| \nabla \bu^{k+1} \right\|^2
	\nonumber \\
	&& \hspace*{-1.0cm}
	\stackrel{\eqref{eqn:pod-filtering-7}}{\leq}
	\widetilde{C}_{15} \, \Delta t \, \sum\limits_{k=0}^{M-1} \left\| \nabla \bu^{k+1} \right\|^2 \ 
		\frac{1}{\delta} \, 
		\biggl[
		\biggl(
			h^{2 m +2} + \Delta t^2 + \sum_{j=r+1}^{d} \lambda_j	
		\biggr)
	\nonumber \\
	&& \hspace*{-1.0cm}
		+ \, \delta^2 \, 
		\biggl( 
			h^{2 m}   
			+ \|S_r\|_2 \, h^{2 m+2}
			+ (1+\|S_r\|_2) \, \Delta t^2
			+\sum _{j = r+1}^{d} \|\bphi_j\|_1^2\, \lambda_j \,
		\biggr)
		+ \, \delta^4 
		\biggr]
	\nonumber \\
	&& \hspace*{-1.0cm}
	\stackrel{\eqref{eqn:continuous-stability}}{\leq}
		C_{15} \, \frac{1}{\delta} \, 
		\biggl[
		\biggl(
			h^{2 m +2} + \Delta t^2 + \sum_{j=r+1}^{d} \lambda_j	
		\biggr)
	\nonumber \\
	&& \hspace*{-1.0cm}
		+ \, \delta^2 \, 
		\biggl( 
			h^{2 m}   
			+ \|S_r\|_2 \, h^{2 m+2}
			+ (1+\|S_r\|_2) \, \Delta t^2
			+\sum _{j = r+1}^{d} \|\bphi_j\|_1^2\, \lambda_j \,
		\biggr)
		+ \, \delta^4 
		\biggr].
	\label{eqn:3.42_rhs_6c}
\end{eqnarray}

Since in \eqref{eq:error-lhs-pressure_1} $q_h$ was an arbitrary function in $Q^h$, we can use the approximation property \eqref{eq:fem_err_p} in Assumption~\ref{lem:fem} to bound the eighth term on the RHS of \eqref{eqn:error-eqn5} as follows: 
\begin{eqnarray}
\Delta t \, \sum\limits_{k=0}^{M-1} \left\| p - q_h \right\|^2
\leq C_{16} \, h^{2 m} \, .
\label{eqn:3.42_rhs_6b}
\end{eqnarray}

Collecting \eqref{eq:proof1}-\eqref{eqn:3.42_rhs_6b}, equation~\eqref{eqn:error-eqn5} becomes
\begin{eqnarray}
	&& 
	\left\| \bPhir^{M} \right\|^2 
	+ C_8 \, \Delta t \sum\limits_{k=0}^{M-1} \left\| \nabla \bPhir^{k+1}  \right\| ^2 
	\nonumber \\
	\leq
	C_9 \, \biggl\{
		&& \Delta t \, \sum\limits_{k=0}^{M-1} \left\| \nabla \bu^{k+1}\right\| ^4 \, \left\| \bPhir^{k+1} \right\| ^2 
	\nonumber \\
	&+& \left(
h^{2m}   
+ \|S_r\|_2 h^{2m+2}
+ \left( 1+\|S_r\|_2 \right) \Delta t^2
+ \sum\limits_{j=r+1}^d \|\bphi_j\|_1^2\, \lambda_j\,
\right)
\nonumber \\
&+& 
	\| S_r \|_2^{1/2}
	\left(
	h^{m+1}   
	+ \Delta t
	+ \sqrt{\sum\limits_{j=r+1}^d \lambda_j } \ 
\right)^2
	+ h^{2 m} 
	+ \Delta t^2
\nonumber 
\end{eqnarray}

\begin{eqnarray}
	&& \hspace*{-1.40cm}
		+ \, \frac{1}{\delta} \, 
		\biggl[
		\biggl(
			h^{2 m +2} + \Delta t^2 + \sum_{j=r+1}^{d} \lambda_j	
		\biggr)
	\nonumber \\
	&& \hspace*{-1.4cm}
		+ \, \delta^2 \, 
		\biggl( 
			h^{2 m}   
			+ \|S_r\|_2 \, h^{2 m+2}
			+ (1+\|S_r\|_2) \, \Delta t^2
			+\sum _{j = r+1}^{d} \|\bphi_j\|_1^2\, \lambda_j \,
		\biggr)
		+ \, \delta^4 
		\biggr]
	\ \biggr\} 
	\nonumber \\
	=
	C_9 \, \biggl\{
		&&  \hspace*{-0.7cm} 
		\Delta t \, \sum\limits_{k=0}^{M-1} \left\| \nabla \bu^{k+1}\right\| ^4 \, \left\| \bPhir^{k+1} \right\| ^2 
	\nonumber \\
	&&  \hspace*{-1.0cm}
	+ \left(
h^{2m}   
+ \|S_r\|_2 h^{2m+2}
+ \left( 1+\|S_r\|_2 \right) \Delta t^2
+ \sum\limits_{j=r+1}^d \|\bphi_j\|_1^2\, \lambda_j\,
\right)
	\nonumber \\
&& \hspace*{-1.0cm} 
	+ \| S_r \|_2^{1/2}
	\left(
	h^{2m+2}   
	+ \Delta t^2
	+ \sum\limits_{j=r+1}^d \lambda_j  \ 
\right)
		+ \frac{1}{\delta} \, 
		\biggl(
			h^{2 m +2} + \Delta t^2 + \sum_{j=r+1}^{d} \lambda_j	
		\biggr)
	\nonumber \\
	&& \hspace*{-1.0cm}
		+ \, \delta \, 
		\biggl( 
			h^{2 m}   
			+ \|S_r\|_2 \, h^{2 m+2}
			+ (1+\|S_r\|_2) \, \Delta t^2
			+\sum _{j = r+1}^{d} \|\bphi_j\|_1^2\, \lambda_j \,
		\biggr)
		+ \, \delta^3 
	\ \biggr\}
	\nonumber \\
	\stackrel{notation}{=}&& 
	C_9 \biggl\{
	\Delta t \, \sum\limits_{k=0}^{M-1} \left\| \nabla \bu^{k+1}\right\| ^4 \, \left\| \bPhir^{k+1} \right\| ^2 
	\nonumber \\
	&& + \cF\biggl( \delta, h, \Delta t, \| S_r \|_2, \{ \lambda_j \}_{j=r+1}^{d}, \{ \|\bphi_j\|_1 \}_{j=r+1}^{d}\biggr)
	\biggr\} \, .
	\label{eq:proof5}
\end{eqnarray}
A discrete Gronwall lemma (see Lemma 27 in \cite{layton2008introduction} and Lemma 5.1 in~\cite{heywood1990finite}) implies that, for small enough $\Delta t$ (i.e., $\Delta t < \bigl( C_{9} \, \max_{1 \leq k \leq M} \| \nabla \bu^{k} \|^4 \bigr)^{-1}$), the following inequality holds:
\begin{eqnarray}
	&& 
	\left\| \bPhir^{M} \right\|^2 
	+ C_8 \, \Delta t \sum\limits_{k=0}^{M-1} \left\| \nabla \bPhir^{k+1}  \right\| ^2 
	\nonumber \\
	&\leq&
	C_{17}  \cF\biggl( \delta, h, \Delta t, \| S_r \|_2, \{ \lambda_j \}_{j=r+1}^{d}, \{ \|\bphi_j\|_1 \}_{j=r+1}^{d}\biggr) \, .
	\label{eq:proof5b}
\end{eqnarray}
By using~\eqref{eq:proof5b}, the triangle inequality, and~\eqref{eq:error_geta}--\eqref{eq:error_eta_no_sum}, yields 

\begin{eqnarray}
&& \left\| {\bf u}^M - {\bf u}_r^M \right\|^2
+ \Delta t \, \sum \limits_{k = 0}^{M-1} \left\| \nabla \left( {\bf u}^{k+1} - {\bf u}_{r}^{k+1}\right) \right\|^2
\nonumber \\
&\leq&
C \, \cF\biggl( \delta, h, \Delta t, \| S_r \|_2, \{ \lambda_j \}_{j=r+1}^{d}, \{ \|\bphi_j\|_1 \}_{j=r+1}^{d}\biggr) \, .
	\label{eqn:proof5c}
\end{eqnarray}
This completes the proof. 
 
\end{proof}

\section{Numerical Results}
	\label{sec:numerical-results}

In this section, we perform a numerical investigation of the theoretical results obtained in Section \ref{sec:error-analysis}. 
To this end, we investigate whether the ROM filtering error estimate in Lemma~\ref{lemma:rom-filtering-error-estimates} and the L-ROM approximation error estimate in Theorem \ref{theorem_error} are recovered numerically.

In the numerical investigation, we consider the same test problem and computational setting as those used in Section IV.B in~\cite{iliescu2014variational}. 
The problem is governed by the 2D incompressible NSE with an analytical solution. 
The exact velocity, $\bu = (u,v)$, has the components $u = \frac{2}{\pi} \arctan(-500(y-t) ) \sin(\pi y)$, $v = \frac{2}{\pi}\arctan(-500(x-t) ) \sin(\pi x)$, and the exact pressure is given by $p = 0$.
The diffusion coefficient is $\nu=10^{-3}$ and the forcing term is chosen to match the exact solution.
The spatial domain $[0, 1]\times [0, 1]$ is discretized by the Taylor-Hood FEs with the mesh size $h=1/64$. 
To generate the POD basis, snapshots are collected over the time interval $[0, 1]$ at every $\Delta T = 10^{-2}$ by recording the exact values of $u$ and $v$ on the FE mesh.  
Following the ansatz in~\eqref{eq:pod}, we do not use the common centering trajectory approach; instead, we apply the method of snapshots to the snapshot data directly. 
The dimension of the POD basis is 101.

\subsection{ROM Filtering Error}
	\label{sec:rom-filtering-error}

In this section, we perform a numerical investigation of the ROM filtering error~\eqref{eqn:pod-filtering-7} in Lemma~\ref{lemma:rom-filtering-error-estimates}.
We define the following average squared filtering errors:
$
\cE_{L^2}
=\frac{1}{M+1} \, \sum \limits_{k=0}^{M} \| \bu^{k} - \overline{\bu^{k}}^r \|^2, 
\cE_{H^1}
=
\frac{1}{M+1} \, \sum \limits_{k=0}^{M} \| \nabla (\bu^{k} - \overline{\bu^{k}}^r) \|^2.
$
The ROM filtering error bound~\eqref{eqn:pod-filtering-7} in Lemma~\ref{lemma:rom-filtering-error-estimates} depends on the parameters $h, \Delta t, \delta$ as well as the ROM truncation errors
$
	\Lambda_{L^2}^r
	= \sum \limits_{j=r+1}^{d} \lambda_j , \, 
	\Lambda_{H^1}^r
	= \sum \limits_{j=r+1}^{d} \| \bphi_j \|_{1}^{2} \, \lambda_j \, .
$
We numerically investigate the rates of convergence of $\cE_{L^2}$ and $\cE_{H^1}$ with respect to the time step $\Delta t$, filter radius $\delta$, and ROM truncation error $\Lambda_{H^1}^r$. 

First, we investigate the convergence rates with respect to $\delta$.
To this end, we fix $h=1/64$, $r=95$ and $\Delta t = 10^{-4}$ (note that the time step size should not matter in this test because the snapshots are FE interpolants of the exact solutions), and vary $\delta$.
With these choices, $h^{2 m} = \cO(10^{-8})$, $\Lambda_{L^2}^r = \cO(10^{-8})$, $\Lambda_{H^1}^r = \cO(10^{-3})$ and $\| S_r \|_2 = \cO(10^{5})$.
Thus, the theoretical error estimate~\eqref{eqn:pod-filtering-7} in Lemma~\ref{lemma:rom-filtering-error-estimates} yields the following asymptotic scaling:
\begin{eqnarray}
	\cE_{L^2}
	&\stackrel{\eqref{eqn:pod-filtering-7}}{\sim}& \cO(\delta^2).
	\label{eqn:rom-filtering-numerical-5} 
\end{eqnarray}
Note that \eqref{eqn:pod-filtering-7} does not provide a scaling between $\cE_{H^1}$ and $\delta$.

We apply the DF~\eqref{eqn:rom-df} to the snapshot data and measure the numerical errors $\cE_{L^2}$ and $\cE_{H^1}$, which are listed in Table~\ref{tab:rom-filtering-numerical-1-test3}. 
Linear regressions of the errors, which are plotted in Fig.~\ref{fig:rom-filtering-numerical-1} for decreasing $\delta$ values, show the following scalings for the ROM filtering errors:
\begin{eqnarray}
	\cE_{L^2}
	&\sim& \cO(\delta^{2.52})
	\label{eqn:rom-filtering-numerical-7} \\
	\cE_{H^1}
	&\sim& \cO(\delta^{1.96}) \, .
	\label{eqn:rom-filtering-numerical-8} 
\end{eqnarray}
Thus, the theoretical scaling~\eqref{eqn:rom-filtering-numerical-5} is numerically recovered. 
On the other hand, although not verified theoretically in~\eqref{eqn:pod-filtering-7}, we do observe the almost quadratic convergence of $\cE_{H^1}$ with respect to $\delta$ (see Remark~\ref{remark:rom-filtering-error-estimates}). 

\begin{table}[h]
\centering
\begin{tabular}{|c|c|c|}
\hline
& &  \\[-0.2cm]
$\delta$ & $\cE_{L^2}$ & $\cE_{H^1}$ \\
& &  \\[-0.2cm]
\hline
$1\times 10^{-2}$&$3.54\times 10^{-3}$&$9.87\times 10^1$ \\
\hline
$5\times 10^{-3}$&$9.14\times 10^{-4}$&$4.65\times 10^1$ \\
\hline
$2.5\times 10^{-3}$&$1.63\times 10^{-4}$&$1.22\times 10^1$ \\
\hline
$2.0\times 10^{-3}$&$8.41\times 10^{-5}$&$6.79\times 10^{0}$ \\
\hline
$1.67\times 10^{-3}$&$4.71\times 10^{-5}$&$3.97\times 10^{0}$ \\
\hline
$1.25\times 10^{-3}$&$1.77\times 10^{-5}$&$1.56\times 10^0$ \\
\hline
\end{tabular} \\
\caption{
	Average ROM filtering errors $\cE_{L^2}$ and $\cE_{H^1}$ for decreasing $\delta$ values.
\label{tab:rom-filtering-numerical-1-test3}
}
\end{table}

\begin{figure}[h]
\centering
\includegraphics[width=.45\textwidth]{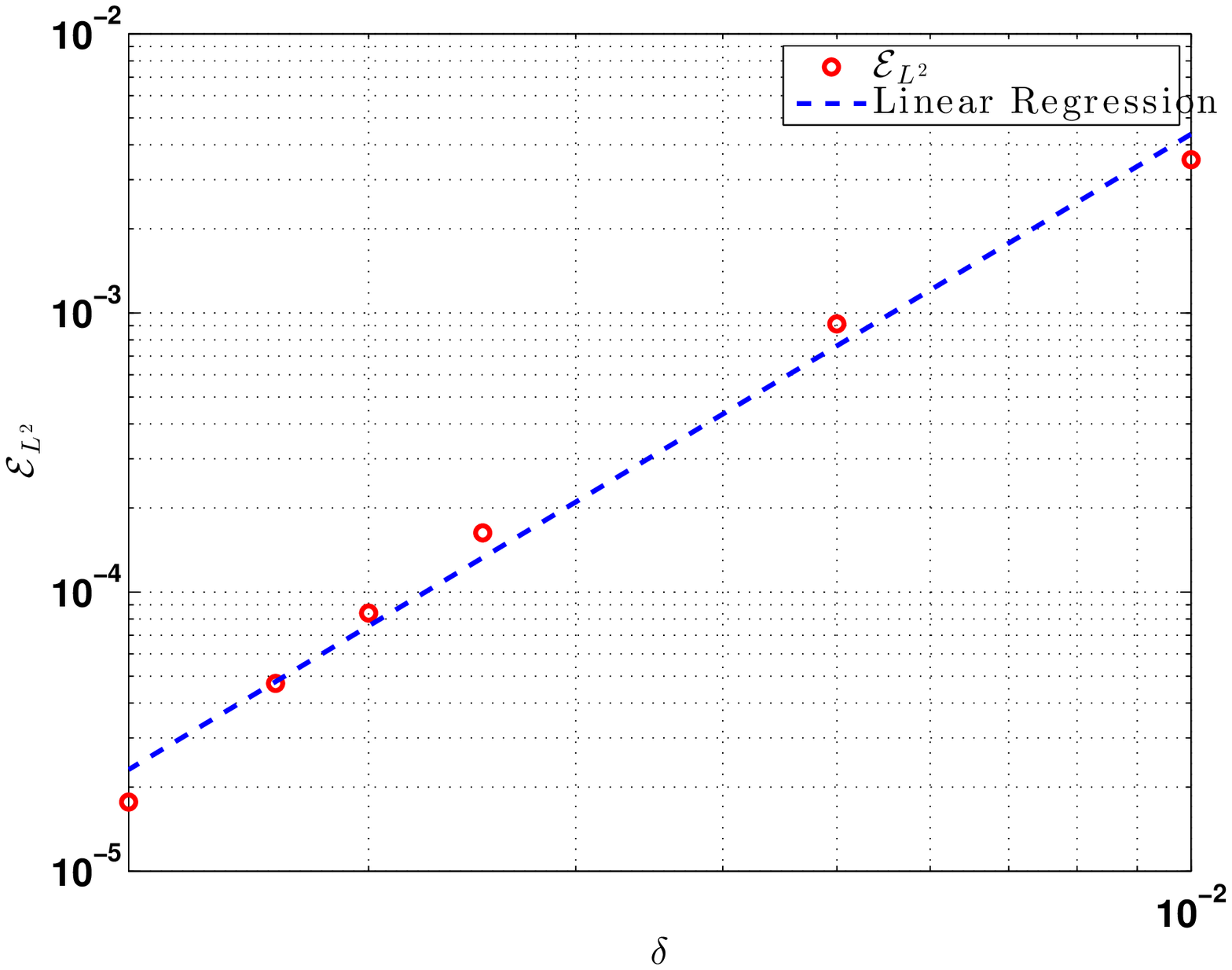}
\hfill
\includegraphics[width=.45\textwidth]{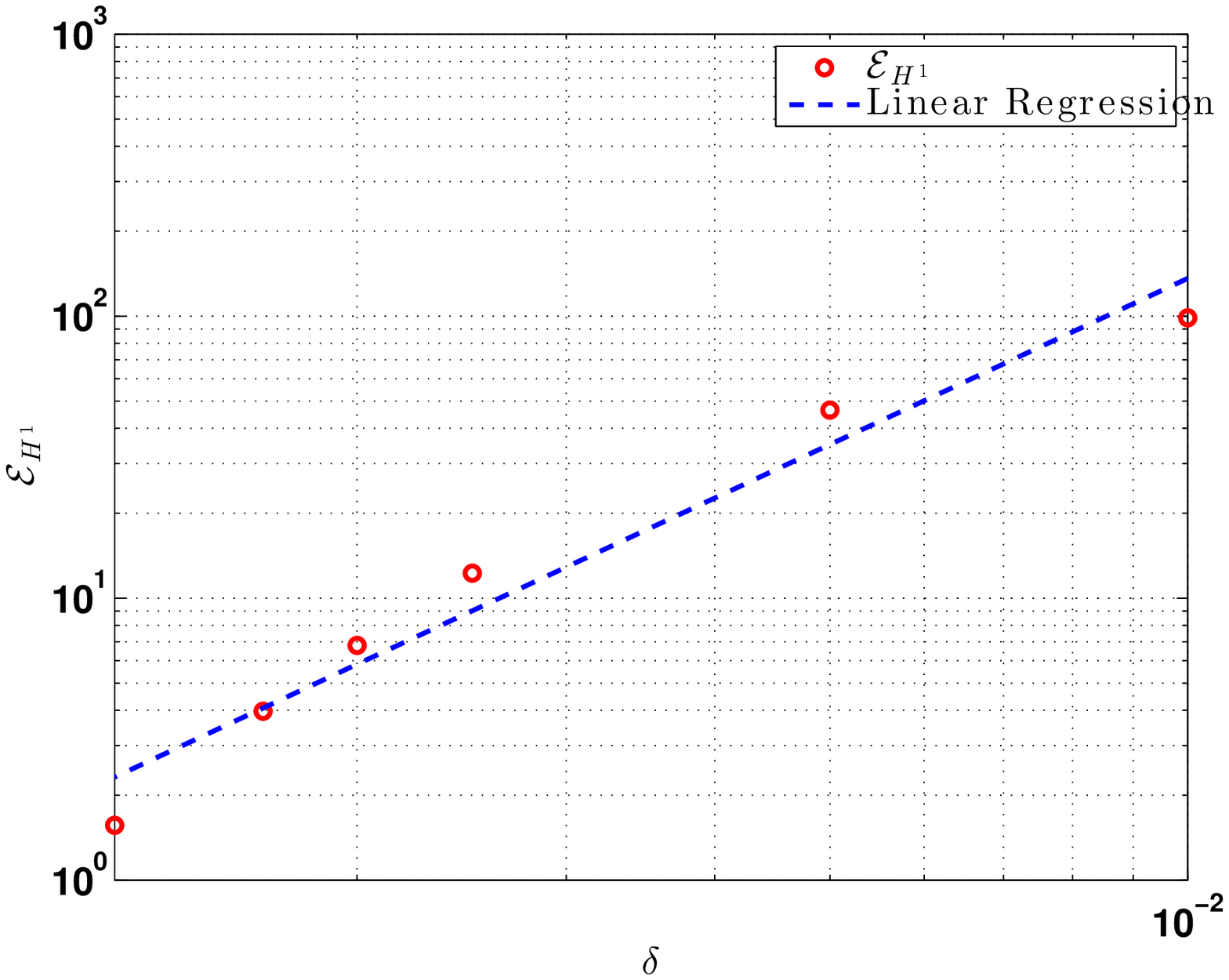}
\caption{Linear regression of $\cE_{L^2}$ and $\cE_{H^1}$ with respect to $\delta$.}
\label{fig:rom-filtering-numerical-1} 
\end{figure}

Next, we investigate the convergence rates with respect to $\Lambda_{H^1}^r$.
To this end, we fix $h=1/64$, $\Delta t = 10^{-4}$, $\delta = 10^{-3}$ and vary $r$.
With these choices, $h^{2 m} = \cO(10^{-8})$, $\delta^2 = \cO(10^{-6})$ and $\| S_r \|_2 = \cO(10^{4})-\cO(10^{5})$.
Thus, the theoretical error estimate~\eqref{eqn:pod-filtering-7} in Lemma~\ref{lemma:rom-filtering-error-estimates} yields the following asymptotic scalings:
\begin{eqnarray}
	\cE_{L^2}
	&\stackrel{\eqref{eqn:pod-filtering-7}}{\sim}& \cO(\Lambda_{H^1}^r)
	\label{eqn:rom-filtering-numerical-9} \\
	\cE_{H^1}
	&\stackrel{\eqref{eqn:pod-filtering-7}}{\sim}& \cO(\Lambda_{H^1}^r) \, .
	\label{eqn:rom-filtering-numerical-10} 
\end{eqnarray}

\begin{table}[h]
\centering
\begin{tabular}{|c|c|c|c|}
\hline
& & & \\[-0.2cm]
$r$ & $\Lambda_{H^1}^r$ & $\cE_{L^2}$ & $\cE_{H^1}$ \\
& & & \\[-0.2cm]
\hline
30 & $1.23\times 10^2$ & $3.29\times 10^{-3}$ & $1.23\times 10^{2}$\\
\hline
40 & $9.26\times 10^1$ & $1.70\times 10^{-3}$ & $9.27\times 10^{1}$\\
\hline
50 & $6.73\times 10^1$ & $9.05\times 10^{-4}$& $6.74\times 10^{1}$\\
\hline
60 & $4.44\times 10^1$ & $4.91\times 10^{-4}$& $4.46\times 10^{1}$\\
\hline
70 & $2.09\times 10^1$ & $2.39\times 10^{-4}$ & $2.14\times 10^{1}$\\
\hline
80 & $6.42\times 10^0$ & $8.11\times 10^{-5}$ &$7.06\times 10^{0}$\\
\hline
\end{tabular} \\
\caption{
	Average ROM filtering errors $\cE_{L^2}$ and $\cE_{H^1}$ for increasing $r$ values.
\label{tab:rom-filtering-numerical-2}
}
\end{table}

\begin{figure}[h]
\centering
\includegraphics[width=.45\textwidth]{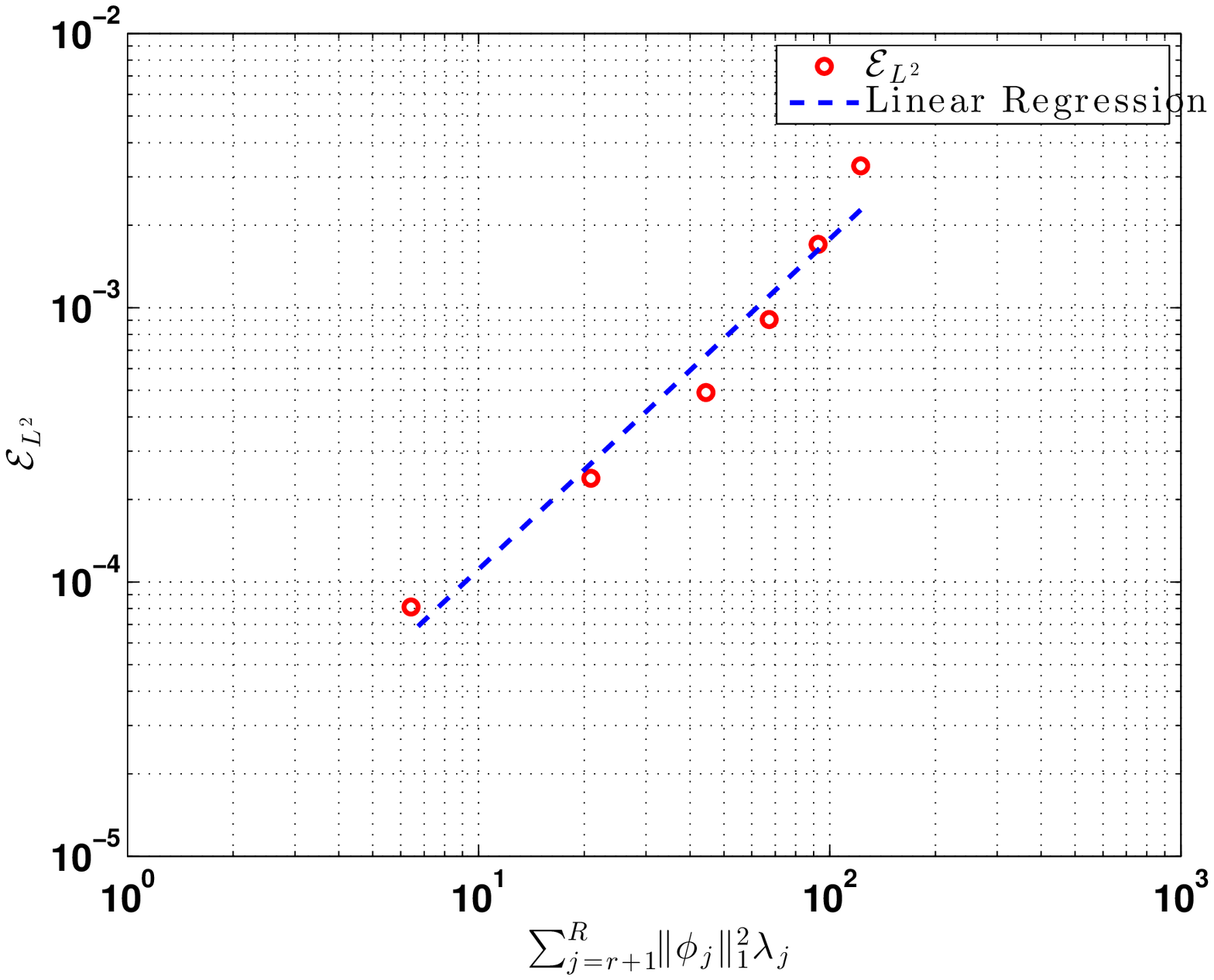}
\hfill
\includegraphics[width=.45\textwidth]{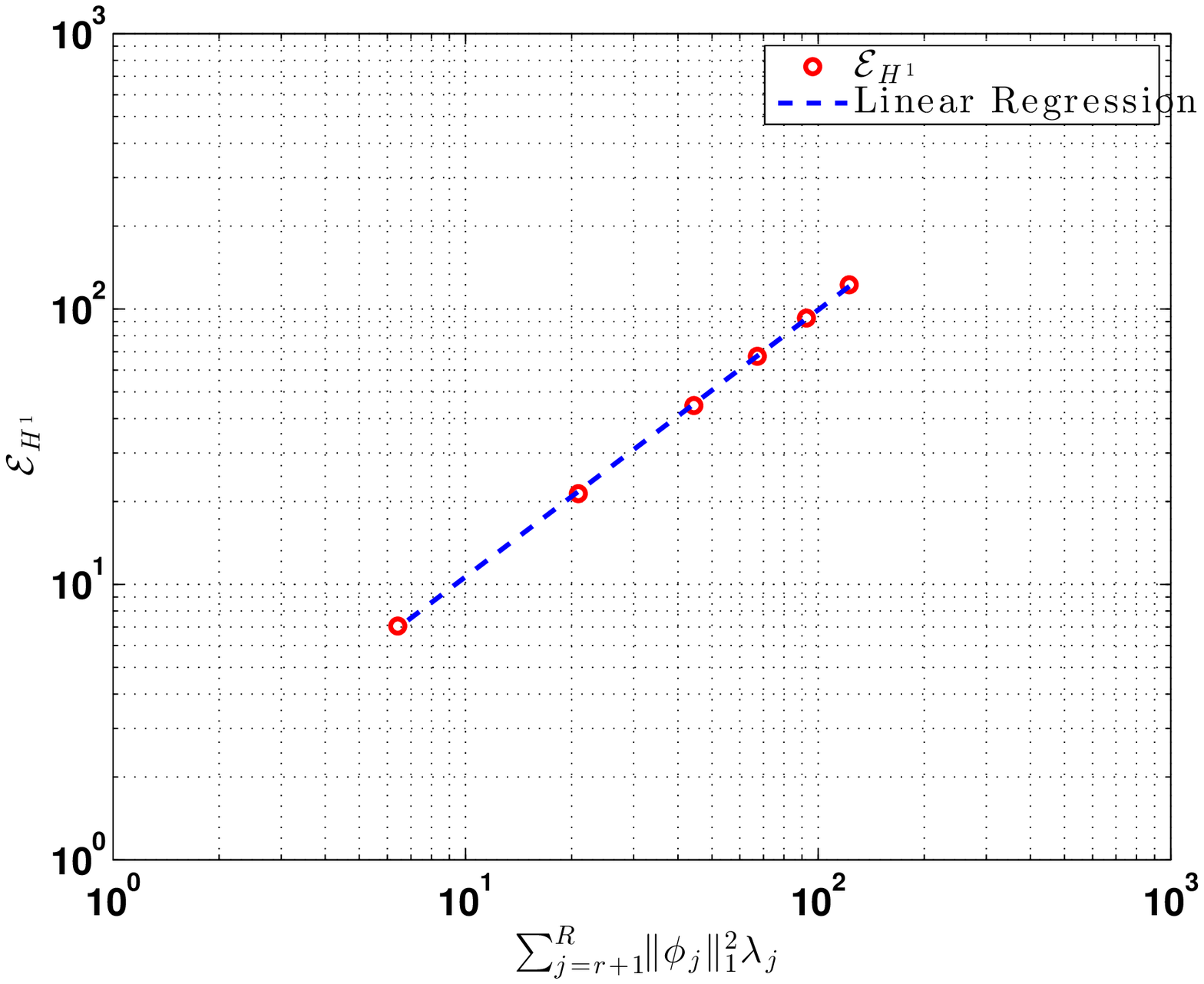}
\caption{Linear regression of $\cE_{L^2}$ and $\cE_{H^1}$ with respect to $\Lambda_{H^1}^r$.}
\label{fig:rom-filtering-numerical-2} 
\end{figure}

The numerical errors $\cE_{L^2}$ and $\cE_{H^1}$ are listed in Table~\ref{tab:rom-filtering-numerical-2} for increasing $r$ values.
The corresponding linear regressions, which are shown in Fig.~\ref{fig:rom-filtering-numerical-2}, indicate the following scalings between the average ROM filtering errors and the ROM truncation error:
\begin{eqnarray}
	\cE_{L^2}
	&\sim& \cO\left((\Lambda_{H^1}^r)^{1.20}\right)
	\label{eqn:rom-filtering-numerical-11} \\
	\cE_{H^1}
	&\sim& \cO\left((\Lambda_{H^1}^r)^{0.97}\right) \, .
	\label{eqn:rom-filtering-numerical-12} 
\end{eqnarray}
Thus, the theoretical scalings~\eqref{eqn:rom-filtering-numerical-9} and~\eqref{eqn:rom-filtering-numerical-10} are numerically recovered.

\subsection{L-ROM Approximation Error}
	\label{sec:l-rom-approximation-error}

In this section, we perform a numerical investigation of the L-ROM approximation error estimate~\eqref{eq:theorem_error_1} in Theorem \ref{theorem_error}.
The L-ROM approximation error at the final time step is
$
\cE_{L^2}^M
= \| \bu^{M} - \overline{\bu^{M}}^r \| \, .
$
We numerically investigate the rates of convergence of $\cE_{L^2}^M$ with respect to the time step $\Delta t$, filter radius $\delta$ and ROM truncation error $\Lambda_{H^1}^r$. 
To this end, we first note that in our numerical investigation $\delta < 1$ and $\| S_r \|_2 \gg 1$.
Thus, the L-ROM approximation error bound~\eqref{eqn:theorem-error-1b} in Theorem \ref{theorem_error} simplifies to the following:
\begin{eqnarray}
	&& \cF\biggl( \delta, h, \Delta t, \| S_r \|_2, \{ \lambda_j \}_{j=r+1}^{d}, \{ \|\bphi_j\|_1 \}_{j=r+1}^{d}\biggr)
	\nonumber \\
	&=& 
h^{2m}   
+ \|S_r\|_2 h^{2m+2}
+ \|S_r\|_2 \,  \Delta t^2
+\Lambda_{H^1}^r \,
	\nonumber \\
&+& 
	\| S_r \|_2^{1/2} \, \Lambda_{L^2}^r  \ 
		+ \frac{1}{\delta} \, 
		\biggl(
			h^{2 m +2} + \Delta t^2 + \Lambda_{L^2}^r	
		\biggr)
	+ \delta^3 \, .
	\label{eqn:l-rom-approximation-error-1}
\end{eqnarray}
The control parameters in the L-ROM approximation error rates of convergence in~\eqref{eqn:l-rom-approximation-error-1} are $h, \Delta t, r, \delta$.

To determine the L-ROM approximation error rate of convergence with respect to $\Delta t$, we fix $h=1/64$, $r=99$, $\delta = 10^{-4}$ and vary $\Delta t$.
With these choices, $h^{2m}= \cO(10^{-8})$, $\Lambda_{L^2}^r = \cO(10^{-8})$, $\Lambda_{H^1}^r = \cO(10^{-3})$ and $\| S_r \|_2 = \cO(10^{5})$. 
Thus, the theoretical L-ROM approximation error bound~\eqref{eqn:l-rom-approximation-error-1} predicts the following rate of convergence of $\cE_{L^2}^M$ with respect to $\Delta t$:
\begin{equation}
	\cE_{L^2}^M 
	= \cO(\Delta t) \, .
	\label{eqn:l-rom-approximation-error-2}
\end{equation}
The L-ROM approximation error $\cE_{L^2}^M$ is listed in Table \ref{L-ROM-1} for decreasing $\Delta t$ values.
The corresponding linear regression, which is shown in Fig.~\ref{fig:LROM-fig-1}, indicates the following L-ROM approximation error rate of convergence with respect to $\Delta t$:
\begin{equation}
	\cE_{L^2}^M 
	= \cO(\Delta t^{0.99}) \, .
	\label{eqn:l-rom-approximation-error-3}
\end{equation}
Thus, the theoretical rate of convergence~\eqref{eqn:l-rom-approximation-error-2} is numerically recovered.

\begin{minipage}{\textwidth}
\begin{minipage}[b]{0.49\linewidth}
\centering
\begin{tabular}{|c|c|}
\hline
$\Delta t$&$\cE_{L^2}^M$\\
\hline
\multirow{1}{*}{$1\times10^{-2}$}      & $2.36\times 10^{-2}$\\
\multirow{1}{*}{$5\times10^{-3}$}      & $2.33\times 10^{-2}$\\
\multirow{1}{*}{$2.5\times10^{-3}$}    & $6.49\times 10^{-3}$\\
\multirow{1}{*}{$1.25\times10^{-3}$}  & $3.49\times 10^{-3}$ \\
\multirow{1}{*}{$6.25\times10^{-4}$}  & $1.96\times 10^{-3}$ \\
\hline
\end{tabular}
\captionof{table}{
	L-ROM approximation error $\cE_{L^2}^M$ for decreasing $\Delta t$ values.
\label{L-ROM-1}}
\end{minipage}
\hfill
\begin{minipage}[b]{0.45\linewidth}
\centering
\includegraphics[width=1\textwidth]{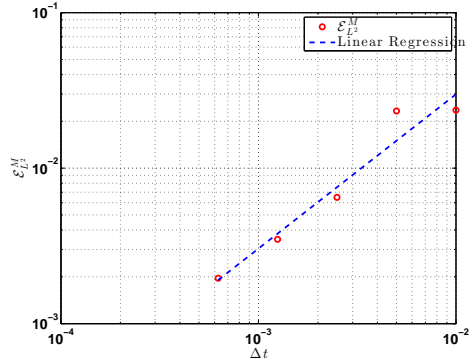}
\captionof{figure}{Linear regression of $\cE_{L^2}^M$ with respect to $\Delta t$.}
\label{fig:LROM-fig-1} 
\end{minipage}
\end{minipage}

\bigskip

To determine the L-ROM approximation error rate of convergence with respect to $\delta$, we fix $h=1/64$, $r=99$, $\Delta t = 10^{-4}$ and vary $\delta$.
With these choices, $h^{2m}= \cO(10^{-8})$, $\Delta t^2 = \cO(10^{-8})$, $\Lambda_{L^2}^r = \cO(10^{-8})$, $\Lambda_{H^1}^r = \cO(10^{-3})$ and $\| S_r \|_2 = \cO(10^{5})$. 
Thus, the theoretical L-ROM approximation error bound~\eqref{eqn:l-rom-approximation-error-1} predicts the following rate of convergence of $\cE_{L^2}^M$ with respect to $\delta$: 
\begin{equation}
	\cE_{L^2}^M 
	= \cO(\delta^{3/2}) \, .
	\label{eqn:l-rom-approximation-error-4}
\end{equation}
The L-ROM approximation error $\cE_{L^2}^M$ is listed in Table \ref{L-ROM-22} for decreasing $\delta$ values.
The corresponding linear regression, which is shown in Fig.~\ref{fig:LROM-fig-22}, indicates the following L-ROM approximation error rate of convergence with respect to $\delta$:
\begin{equation}
	\cE_{L^2}^M 
	= \cO(\delta^{2.09}) \, .
	\label{eqn:l-rom-approximation-error-5}
\end{equation}
We note that the numerical rate of convergence~\eqref{eqn:l-rom-approximation-error-5} is higher than the theoretical rate of convergence~\eqref{eqn:l-rom-approximation-error-4} (see Remark~\ref{remark:rom-filtering-error-estimates}).

\begin{minipage}{\textwidth}
  \begin{minipage}[b]{0.49\textwidth}
\centering
\begin{tabular}{|c|c|}
\hline
$\delta$&$\cE_{L^2}^M$\\
\hline
\multirow{1}{*}{$5\times 10^{-1}$}      & $8.47\times 10^{-1}$\\
\multirow{1}{*}{$2.5\times 10^{-1}$}      & $4.15\times 10^{-1}$\\
\multirow{1}{*}{$1.25\times 10^{-1}$}   & $1.14\times 10^{-1}$\\
\multirow{1}{*}{$6.25\times 10^{-2}$}   & $1.96\times 10^{-2}$\\
\multirow{1}{*}{$3.12\times 10^{-2}$}  & $2.81\times 10^{-3}$\\
\multirow{1}{*}{$1.56\times 10^{-2}$}  & $9.59\times 10^{-4}$\\
\hline
\end{tabular}
\captionof{table}{
	L-ROM approximation error $\cE_{L^2}^M$ for decreasing $\delta$ values.
\label{L-ROM-22}}
\end{minipage}
\hfill
\begin{minipage}[b]{0.45\linewidth}
\centering
\includegraphics[width=1\textwidth]{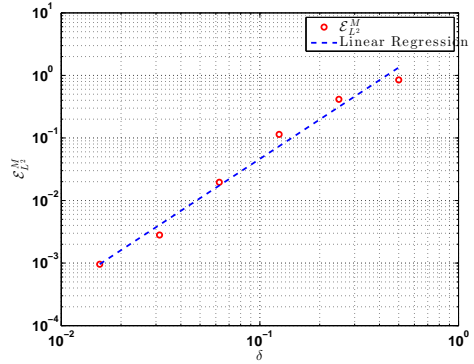}
\captionof{figure}{Linear regression of $\cE_{L^2}^M$ with respect to $\delta$.}
\label{fig:LROM-fig-22} 
\end{minipage}
\end{minipage}

\bigskip

To determine the L-ROM approximation error rate of convergence with respect to $\Lambda^r_{H^1}$, we fix $h=1/64$, $\Delta t = 10^{-4}$, $\delta = 10^{-2}$ and vary $r$.
With these choices, $h^{2m}= \cO(10^{-8})$ and $\| S_r \|_2 = \cO(10^{2})-\cO(10^{5})$. 
Thus, the theoretical L-ROM approximation error bound~\eqref{eqn:l-rom-approximation-error-1} predicts the following rate of convergence of $\cE_{L^2}^M$ with respect to $\Lambda^r_{H^1}$:
\begin{equation}
	\cE_{L^2}^M
	= \cO(\Lambda^r_{H^1}) \, .
	\label{eqn:l-rom-approximation-error-6}
\end{equation}
The L-ROM approximation error $\cE_{L^2}^M$ is listed in Table \ref{L-ROM-3} for increasing $r$ values.
The corresponding linear regression, which is shown in Fig.~\ref{fig:LROM-fig-2}, indicates the following L-ROM approximation error rate of convergence with respect to $\Lambda^r_{H^1}$:
\begin{equation}
	\cE_{L^2}^M 
	= \cO((\Lambda^r_{H^1})^{1.53}) \, .
	\label{eqn:l-rom-approximation-error-7}
\end{equation}
Thus, the theoretical rate of convergence~\eqref{eqn:l-rom-approximation-error-6} is numerically recovered.

\begin{minipage}{\textwidth}
  \begin{minipage}[b]{0.49\textwidth}
    \centering
    \begin{tabular}{|c|c|c|}
	\hline
	r&$\Lambda^r_{H^1}$&$\cE_{L^2}^M$\\
	\hline
	10& \multirow{1}{*}{$1.99\times10^{2}$}      & $9.62\times 10^{-2}$\\
	20& \multirow{1}{*}{$1.57\times10^{2}$}      & $5.15\times 10^{-2}$\\
	30& \multirow{1}{*}{$1.22\times10^{2}$}      & $3.05\times 10^{-2}$\\
	40& \multirow{1}{*}{$9.26\times10^{1}$}    & $2.09\times 10^{-2}$\\
	50& \multirow{1}{*}{$6.73\times10^{1}$}  & $1.83\times 10^{-2}$\\
	\hline
    \end{tabular}
	\captionof{table}{
	L-ROM approximation error $\cE_{L^2}^M$ for increasing $r$ values.
	\label{L-ROM-3}}
  \end{minipage}
  \hfill
  \begin{minipage}[b]{0.45\textwidth}
    \centering
	\includegraphics[width=1\textwidth]{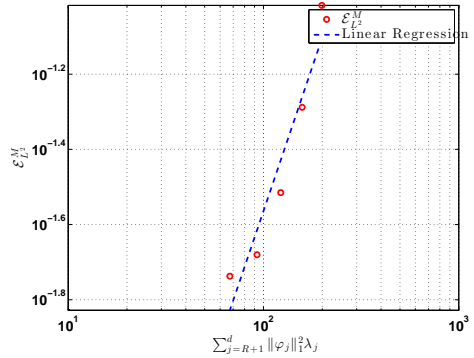}
	\captionof{figure}{Linear regression of $\cE_{L^2}^M$ with respect to $\Lambda^r_{H^1}$.}
	\label{fig:LROM-fig-2} 
    \end{minipage}
\end{minipage}

\section{Conclusions and Outlook}
	\label{sec:conclusions}
	
Several modeling strategies have been proposed to alleviate the spurious numerical oscillations that generally appear 	when standard ROMs are used to simulate convection-dominated flows.
Reg-ROMs are recently proposed ROMs in which numerical stabilization is achieved through explicit ROM spatial filtering.
Reg-ROMs were successfully used in~\cite{wells2017evolve} in the numerical simulation of the 3D flow past a circular cylinder at a Reynolds number $Re = 1000$.
Reg-ROMs were also employed in~\cite{iliescu2017regularized} for the stabilization of ROMs in the numerical simulation of a stochastic Burgers equation.
To our knowledge, however, there is no numerical analysis of the Reg-ROMs and the explicit ROM spatial filter.

In this paper, we took a first step in this direction and, in Theorem~\ref{theorem_error}, we proved error estimates for the FE discretization of one such Reg-ROM, the L-ROM~\cite{iliescu2017regularized,sabetghadam2012alpha,wells2017evolve}.
In Lemma~\ref{lemma:rom-filtering-error-estimates}, we also proved error estimates for the FE discretization of  the ROM differential filter, which is the explicit ROM spatial filter that we used in the construction of the L-ROM.
Finally, in Section~\ref{sec:numerical-results}, we provided a numerical verification of the ROM filtering error estimate derived in Lemma~\ref{lemma:rom-filtering-error-estimates} and the L-ROM approximation error estimate in Theorem \ref{theorem_error}.
In our numerical investigation, we considered the 2D incompressible NSE with an analytical solution and small diffusion coefficient $\nu=10^{-3}$, which is the computational setting used in~\cite{iliescu2014variational}. 

There are several research directions that could be pursued.
As noted in Remark~\ref{remark:rom-filtering-error-estimates}, one could try to extend from the FE setting to the ROM setting Lemma 2.4 in~\cite{dunca2013mathematical} instead of Lemma 2.12 in~\cite{layton2008numerical}, as we did in Lemma~\ref{lemma:rom-filtering-error-estimates}, since the former could yield better $\delta$ scalings of the $H^1$ seminorm of the filtering error.
However, one would probably first have to prove the $H^1$ stability of the ROM $L^2$ projection, which, to our knowledge, has not been achieved yet.
Another research direction is the extension of the numerical analysis for the L-ROM to other Reg-ROMs, such as the evolve-then-filter ROM~\cite{wells2017evolve}.
Finally, one could also try to prove error estimates for the novel large eddy simulation ROMs introduced in~\cite{xie2017approximate}, in which the explicit ROM filter error plays a central role.

\section*{Acknowledgements}

The authors greatly appreciate the financial support of the National Science Foundation through grants DMS-1522656 and DMS-1522672.

\bibliographystyle{plain}
\bibliography{traian}

\end{document}

%% file: notation.tex
\newcommand{\lp}{\left(}
\newcommand{\rp}{\right)}

\newcommand{\lnorm}{\left\|}
\newcommand{\rnorm}{\right\|}

\newtheorem{remark}{Remark}[section]
\newtheorem{lemma}{Lemma}[section]
\newtheorem{theorem}{Theorem}[section]

\newtheorem{proposition}{Proposition}[section]
\newtheorem{definition}{Definition}[section]
\newtheorem{assumption}{Assumption}[section]

\def\PP{{{\rm l}\kern - .15em {\rm P} }}
\def\PN2{{\PP_{N}-\PP_{N-2}}}

\newcommand{\R}{\mathbbm{R}}

\newcommand{\cE}{\mathcal{E}}
\newcommand{\cF}{\mathcal{F}}
\newcommand{\cH}{\mathcal{H}}
\newcommand{\cO}{\mathcal{O}}

\newcommand{\bfeta}{\boldsymbol{\eta}}

\newcommand{\bPhir}{\boldsymbol{\Phi}_r}
\newcommand{\bphi}{\boldsymbol{\varphi}}

\newcommand{\be}{\boldsymbol{e}}
\newcommand{\bff}{\boldsymbol{f}}

\newcommand{\bH}{\boldsymbol{H}}
\newcommand{\bL}{\boldsymbol{L}}

\newcommand{\br}{\boldsymbol{r}}

\newcommand{\bu}{\boldsymbol{u}}

\newcommand{\bur}{{\boldsymbol{u}}_r}

\newcommand{\bv}{\boldsymbol{v}}
\newcommand{\bV}{\boldsymbol{V}}
\newcommand{\bvr}{{\boldsymbol{v}}_r}
\newcommand{\bw}{\boldsymbol{w}}

\newcommand{\bwr}{{\boldsymbol{w}}_r}

\newcommand{\bx}{\boldsymbol{x}}
\newcommand{\bX}{\boldsymbol{X}}

\newcommand{\bXr}{{\bf X}^r}

\newcommand{\bz}{\boldsymbol{z}}

\newcommand{\obur}{\overline{\boldsymbol{u}_r}}

\newcommand{\obr}[1]{\overline{\boldsymbol{#1}}^r}

\newcommand{\CinvNabla}{C_{inv}^{\nabla}(r)}

\newcommand{\Deltar}{\Delta_r}

\newcommand{\deleted}[1]{{}}

\newcommand{\half}{\frac{1}{2}}